\documentclass[12pt,centertags,oneside]{amsart}
\usepackage{a4}
\usepackage{amsmath,amstext,amsthm,amssymb,amsfonts,amscd}
\usepackage{mathrsfs}
\usepackage[colorlinks=false]{hyperref}
\usepackage{graphicx}
\usepackage[T1]{fontenc}
\usepackage[latin1]{inputenc}
\usepackage{color,tocvsec2}
\usepackage{typearea}
\usepackage{charter}
\usepackage{enumerate}
\usepackage{lscape}
\usepackage[all]{xy}
\usepackage{tikz}
\usetikzlibrary{matrix,arrows,decorations.pathmorphing}

\usepackage{multicol}        
\makeindex             

\usepackage[a4paper,width=16.2cm,top=2.5cm,bottom=2.5cm]{geometry}

\theoremstyle{plain}
\newtheorem{lemma}{Lemma}[section]
\newtheorem{thm}[lemma]{Theorem}
\newtheorem{prop}[lemma]{Proposition}

\theoremstyle{definition}
\newtheorem{defn}[lemma]{Definition}

\theoremstyle{remark}
\newtheorem{rem}[lemma]{Remark}

\numberwithin{equation}{section}
\numberwithin{figure}{section}

\newcommand{\mA}{\mathcal{A}}
\newcommand{\mF}{\mathcal{F}}
\newcommand{\mH}{\mathcal{H}}
\newcommand{\cC}{\mathcal{C}}

\newcommand{\N}{\mathbb{N}}
\newcommand{\Z}{\mathbb{Z}}
\newcommand{\R}{\mathbb{R}}
\newcommand{\C}{\mathbb{C}}

\DeclareMathOperator{\tr}{\mathrm Tr}
\DeclareMathOperator{\ind}{\mathrm Ind}
\DeclareMathOperator{\ch}{\mathrm ch}

\DeclareMathOperator{\Id}{\mathrm Id}
\DeclareMathOperator{\Ind}{\mathrm Ind}
\DeclareMathOperator{\Ahat}{ \widehat{\mathrm{A}}}
\DeclareMathOperator{\Sf}{\mathrm{sf}}

\newcommand{\wi}{\widetilde}

\newcommand{\ul}{\underline}

\begin{document}

\title{Bismut-Cheeger eta form and higher spectral flow}
\author{Bo Liu}
\address{School of Mathematical Sciences,
	Shanghai  Key Laboratory of PMMP,
	East China Normal University, 
	Shanghai, 200241
	People's Republic of China}
\email{bliu@math.ecnu.edu.cn}
\thanks {}

\date{\today}

\dedicatory{}

\begin{abstract}
We prove the variation formula, embedding formula
and the adiabatic limit formula for equivariant Bismut-Cheeger eta forms with a fiberwise compact Lie group action. We only need the condition that the kernels of corresponding 
fiberwise Dirac operators form vector bundles. 
\end{abstract}

\maketitle

\settocdepth{section}
\section{Introduction}\label{s00}

The Atiyah-Patodi-Singer eta invariant is the boundary contribution
of the index theorem for manifold with boundary with the global boundary condition \cite{APS73}. 
Its family extension, the Bismut-Cheeger eta form, 
is well-defined for a fibration
of spin manifolds when the kernel of the fiberwise Dirac
operator forms a vector bundle \cite{BC89}. Moreover, the Bismut-Cheeger eta form can be
naturally extended to the equivariant case for a fiberwise compact Lie group action.
 If the base manifold of the 
fibration is a point and the group is trivial, the equivariant eta form degenerates to the eta invariant.
So the equivariant eta form can be considered as the equivariant higher degree version of the eta invariant.

The purpose of this paper is to generalize some properties of
eta invariants: variation formula, embedding formula and 
the adiabatic limit formula, to equivariant Bismut-Cheeger 
eta forms. 
For the variation formula and the adiabatic limit formula,
we remove additional technical assumptions in previous
paper \cite{Liu17}.

In \cite{Embed,Liu21}, the author proved these properties for the 
equivariant eta form with perturbation, which is the equivariant
version of the modified eta form in \cite{MP97a,MP97b}.
In this paper, we obtain
these formulas for the equivariant version of the original
Bismut-Cheeger eta form by comparing these two versions of equivariant
eta forms and the analysis of the 
equivariant Dai-Zhang higher spectral flow.
%

\subsection{Properties of reduced eta invariants}\label{s0001}

Let $(X,g^{TX})$ be an odd dimensional closed spin manifold. 
Let 
$\ul{E}:=(E, h^E, \nabla^E)$ be a triple which consists of
a complex vector bundle $E$ over $X$, a Hermitian metric
$h^E$ on $E$ and a connection $\nabla^E$ preserving $h^E$. We call $\ul{E}$ a geometric triple.
Let $D_X^E$ be the Dirac operator twisted with these geometric 
data. 
The Atiyah-Patodi-Singer eta invariant of the Dirac operator
is defined by
\begin{align}\label{eq:0.01}
\eta\left(D_X^E\right):=\int_0^{\infty}\tr\left[D_X^E\exp
\left(-u(D_X^E)^2 \right) \right]\frac{du}{\sqrt{\pi u}},
\end{align}
which is a global spectral invariant. The reduced version
\begin{align}\label{eq:0.02}
\bar{\eta}\left(D_X^E\right):=\frac{1}{2}\eta\left(D_X^E\right)+
\frac{1}{2}\dim \ker D_X^E
\end{align} 
appears as the boundary term in the index theorem of 
Atiyah-Patodi-Singer \cite{APS73} for compact manifold with boundary.

Note that the eta invariant depends on the geometric data
$(g^{TX},h^E,\nabla^E)$. So it is not a topological invariant.
But when the geometric data varies, the variation formula
could be written explicitly.

	For $i=0,1$,
let $g_i^{TX}$ be Riemannian metrics on $TX$ and
$\nabla_i^{TX}$ be Levi-Civita connections with respect to $g_i^{TX}$. 
Let $\ul{E_i}=(E,h_i^E, \nabla_i^E)$, $i=0,1$, be two geometrc triples over $X$. 
Let $D_i$ be the Dirac operators associated with the geometric data $(g_i^{TX}, \nabla_i^E)$
respectively.

\begin{thm}\label{thm:0.01} \cite{APS76}
The difference of two reduced eta invariants
\begin{multline}\label{eq:0.03}
\bar{\eta}\left(D_1\right)-\bar{\eta}\left(D_0\right)
=\int_X\wi{\Ahat}\left(TX, \nabla_0^{TX}, \nabla_1^{TX} \right)\ch\left(E, \nabla_0^{E}\right)
\\
+\int_X\Ahat\left(TX, \nabla_1^{TX}\right)\wi{\ch}\left(E,\nabla_0^{E},\nabla_1^{E}\right)
+\Sf\left(D_0, D_1\right).
\end{multline}
Here $\Ahat(\cdot)$, $\ch(\cdot)$ are the $\Ahat$-form and the Chern
character form \cite[\S 1.5]{BGV04}, $\wi{\Ahat}(\cdot)$, $\wi{\ch}(\cdot)$ are  corresponding
Chern-Simons forms \cite[Definition B.5.3]{MM07} and
 $\Sf\left(D_0, D_1\right)\in \Z$ is the spectral flow
 between $D_0$ and $D_1$ \cite[p94]{APS76}. 
\end{thm}

Another important property of the eta invariant is the 
Bismut-Zhang embedding formula \cite{BZ93}, which plays the 
key role in the differential index theorem 
\cite{FL10} and the equivariant version of which \cite{Embed}
is an important component in the proof of the localization formula
of eta invariants \cite{LiuMa20}.

\begin{thm}\label{thm:0.02}\cite{BZ93}
Let $i:Y\to X$ be an 
embedding between two odd dimensional closed 
spin manifolds. Let $g^{TX}$ be a Riemannian metric on $TX$ and $g^{TY}$ be the induced metric on $TY$. 
Let $\nabla^{TY}$ and $\nabla^{TX}$ be corresponding 
Levi-Civita connections.
Let $\ul{\mu}=(\mu, h^{\mu},\nabla^{\mu})$ 
be a geometric triple over $Y$. Then there exist geometric triples
$\ul{\xi_{\pm}}=(\xi_{\pm},h^{\xi,\pm},\nabla^{\xi,\pm})$ over $X$ satisfying the metric condition \cite[(1.10)]{BZ93}
and a spectral flow term $\Sf(Y,X)\in \Z$ such that
\begin{multline}\label{eq:0.04}
\bar{\eta}(D_X^{\xi_+})-\bar{\eta}(D_X^{\xi_-})
=\bar{\eta}(D_Y^{\mu})+\int_X\Ahat(TX, \nabla^{TX})\gamma(Y,X)
\\
+\int_Y\wi{\Ahat}(TY, \nabla^{TY,N},\nabla^{TX|_Y})
\Ahat^{-1}(N,\nabla^N)\ch(E,\nabla^E)+\Sf(Y,X).
\end{multline}
Here $\nabla^N$
is the connection on the normal bundle $N$ of $Y$ in $X$ induced 
from the connection $\nabla^{TX}$,
 $\nabla^{TY,N}:=\nabla^{TY}\oplus\nabla^N$ 
 and $\gamma(Y,X)$ is the Bismut-Zhang current introduced 
 in \cite{BZ93}.
\end{thm}

Remark that the vector bundles $\xi_{\pm}$ over $X$
are the Atiyah-Hirzebruch direct image \cite{AH59} of $\mu$
over $Y$ and the geometric triples $\ul{\xi_{\pm}}$ were constructed in \cite[(1.10)]{BZ93}. In \cite{BZ93}, Bismut and Zhang established this embedding formula as an equation mod $\Z$. 
In \cite[Theorem 4.1]{DZ00}, when $X$ and $Y$ are boundaries, the hidden mod $\Z$ term was explained as some APS index.
In \cite{Embed},
the author explained the mod $\Z$ term as a spectral flow
constructively in general case.



\

Now we introduce the adiabatic limit formula of eta invariants 
established by Bismut-Cheeger \cite{BC89} and Dai \cite{Dai91}.
Let $\pi:W\to B$ be a submersion of two closed manifolds
with fiber $X$.
Let $TX:=\ker (\pi_*:TW\to TB)$ be the relative tangent bundle.
Let $T^HW$ be a horizontal subbundle of $TW$ such that $TW=T^HW\oplus TX$.
Let $g^{TX}$ be a metric on $TX$. 
Let $\nabla^{TX}$ be the connection on $TX$
associated with $(T^HW, g^{TX})$ defined in (\ref{eq:1.01}).
Assume that $TX$ has a spin structure.
Let $\ul{E}$ be a geometric triple over $W$.
Let $D_X^E$ be the fiberwise Dirac operator (see (\ref{eq:1.02})).
Assume that $\ker D_X^E$ forms a vector bundle
over $B$. Under this assumption, the Bismut-Cheeger eta form
$\tilde{\eta}^{BC}
(\pi,E)\in \Omega^*(B)$ (non-equivariant version of Definition \ref{defn:1.01})
is well-defined. Moreover, by \cite[(0.6)]{BC89} and \cite[(0.2)]{Dai91},
\begin{align}\label{eq:0.06}
d\tilde{\eta}^{BC}(\pi,E)=
\left\{
\begin{aligned}
&\int_X\Ahat(TX,\nabla^{TX})\ch(E,\nabla^E)-
\ch(\ker D_X^E, \nabla^{\ker}),  &\hbox{if $\dim X$ is even;}
\\
&\int_X\Ahat(TX,\nabla^{TX})\ch(E,\nabla^E),
\  &\hbox{if $\dim X$ is odd.}
\end{aligned}
\right.
\end{align}
Here  $\nabla^{\ker}$ is the connection on the vector bundle $\ker D_X^E$ defined in (\ref{eq:1.03}).
Let $g^{TB}$ be a Riemannian metric on $TB$.
Assume that $B$ is spin. Then $W$ is also spin. 
For $t>0$, let $g_t^{TW}:=g^{TX}\oplus t^{-2}\pi^*g^{TB}$, which is a Riemannian metric on $TW$.
Let $D_{W,t}^E$ be the Dirac operator associated with $(g_t^{TW}, \nabla^E)$. The following theorem, called the adiabatic limit formula,
was established
by Bismut and Cheeger in \cite{BC89} when $D_X^E$ is invertible
and later extended to the case that $\ker D_X^E$ forms a vector bundle \cite{Dai91}. See also \cite[(54)]{Lo94} and \cite[Theorem 2.5]{Zh94} for some discussions and applications.


\begin{thm}\label{thm:0.03} \cite{BC89,Dai91}
If $\dim W$ is odd,
\begin{align}\label{eq:0.05}
\lim_{t\to 0}\bar{\eta}\left(D_{W,t}^E \right)
\equiv 
\left\{
\begin{aligned}
&\int_B \Ahat(TB, \nabla^{TB})\tilde{\eta}^{BC}(\pi,E)+\bar{\eta}\left(D_B^{\ker D_X^E} \right) \mod \Z,\  &\hbox{if $\dim X$ is even;}
\\
&\int_B \Ahat(TB, \nabla^{TB})\tilde{\eta}^{BC}(\pi,E) \mod \Z,
\  &\hbox{if $\dim X$ is odd.}
\end{aligned}
\right.
\end{align}
\end{thm}

Remark that if $B$ is a point and $\dim X$ is odd, 
\begin{align}\label{eq:0.07}
\tilde{\eta}^{BC}(\pi,E)=\frac{1}{2}\eta(D_X^E).
\end{align}
Thus the Bismut-Cheeger eta form can be considered as the higher degree version of the eta invariant. 

\subsection{Reduced equivariant Bismut-Cheeger eta form}\label{s0002}

Let $G$ be a compact Lie group which acts on $W$ and $B$ such that $\pi\circ g=\pi$ for any $g\in G$. Then the $G$-action on $B$ is trivial. We assume that the $G$-action preserves the splitting $TW=T^HW\oplus TX$ and the spin structure of $TX$. We assume that $E$ is $G$-equivariant and
$g^{TX}, h^E, \nabla^E$ are $G$-invariant.
Then the $G$-action commutes with $D_X^E$.
 
Assume that $\ker D_X^E$ forms a vector
	 bundle over $B$.
Under this assumption, 
the equivariant Bismut-Cheeger eta form $\tilde{\eta}^{BC}_{g}(\pi,E)\in \Omega^*(B)$ is well-defined for any $g\in G$ (Definition \ref{defn:1.01}). For the unity $e\in G$, $\tilde{\eta}^{BC}_e(\pi,E)=\tilde{\eta}^{BC}(\pi,E)$. 
Moreover
by \cite[(1.3)]{Liu17} (cf. also \cite[(0.2)]{Dai91}),
\begin{align}\label{eq:0.08}
d\tilde{\eta}^{BC}_g(\pi,E)=
\left\{
\begin{aligned}
&\int_{X^g}\Ahat_g(TX,\nabla^{TX})\ch_g(E,\nabla^E)-
\ch_g(\ker D_X^E, \nabla^{\ker}),  &\hbox{if $\dim X$ is even;}
\\
&\int_{X^g}\Ahat_g(TX,\nabla^{TX})\ch_g(E,\nabla^E),
\  &\hbox{if $\dim X$ is odd.}
\end{aligned}
\right.
\end{align}
 Here $X^g$ is the fixed point set of $g\in G$ on $X$, $\Ahat_g(\cdot)$ and $\ch_g(\cdot)$ are the equivariant $\Ahat$-form and the equivariant Chern
character form respectively (see e.g., \cite[Definition 1.3]{LiuMa00}, \cite[(2.44), (2,45)]{Liu17} for the definitions).

As in (\ref{eq:0.02}), we define the reduced equivariant eta form
$\tilde{\bar{\eta}}^{BC}_g(\pi, E)$ by
\begin{align}\label{eq:0.10}
\tilde{\bar{\eta}}^{BC}_g(\pi, E)=
\left\{
\begin{aligned}
&\tilde{\eta}^{BC}_g(\pi, E)+\frac{1}{2}\ch_g(\ker D_X^E, \nabla^{\ker})  &\hbox{if $\dim X$ is odd;}
\\
&\tilde{\eta}^{BC}_g(\pi, E),
\  &\hbox{if $\dim X$ is even.}
\end{aligned}
\right.
\end{align}
When $B$ is a point and $\dim X$ is odd, as in (\ref{eq:0.07}), 
we have
\begin{align}\label{eq:0.13}
\tilde{\bar{\eta}}^{BC}_e(\pi, E)=\bar{\eta}(D_X^E).
\end{align}
So the reduced equivariant Bismut-Cheeger eta form is the equivariant higher degree version of the reduced eta invariant.

In this paper, we will generalize Theorems \ref{thm:0.01}-\ref{thm:0.03} for reduced eta invariants to reduced equivariant Bismut-Cheeger eta forms. 
For the analogue of these results for holomorphic torsion 
forms, see \cite[Appendix]{Ta21}.

For the generalization, the main difficulties are
 concentrated on 
the spectral flow terms. In order to generalize these terms, we need to replace the spectral flow terms by the equivariant higher spectral flow terms. Unfortunately, until now the equivariant higher spectral flow is only well-defined when the family index vanishes as an element of 
the equivariant $K$-group of the base space (see \cite[Definition 1.5]{DZ98} and \cite[Definitions 3.7, 3.8]{Liu21}). In general, 
the assumption that the 
kernel of the fiberwise Dirac operator forms a
 vector bundle
cannot guarantee that the family index vanishes. 

In this paper, we use a trick to overcome this difficulty by
adding new vector bundles to make the equivariant family index vanish. 
In this case, the equivariant eta form with perturbation
is well-defined (Definition \ref{defn:1.02}).
Then we establish a comparison formula between it and the reduced equivariant
eta form (\ref{eq:0.10}).
And
the generalizations follow from the results for equivariant eta
forms with perturbation in \cite{Liu21} and the analysis of the
higher spectral flow. This trick is inspired
by the perturbation techniques in \cite{FL10,Lo94}.

Although we use the spin condition here,
all results in this paper can be naturally extended to the 
equivariant Clifford module case.


In the following subsections, we will describe our results in 
more details.

\subsection{Variation formula}\label{s0003}

Let $\pi:W\to B$ be an equivariant submersion of two closed 
$G$-manifolds with closed fiber $X$ such that 
$G$ acts on $B$ trivially and the relative tangent
bundle $TX$ has an equivariant spin structure.
Let $T_0^HW$ and $T_1^HW$ be two equivariant horizontal subbundles.
Let $g_0^{TX}$ and $g^{TX}_1$ be two $G$-invariant Riemannian metrics on $TX$.
Let $\nabla_0^{TX}$ and $\nabla_1^{TX}$ be connections in (\ref{eq:1.01}) associated with $(T_0^HW, g_0^{TX})$ and 
$(T_1^HW, g_1^{TX})$.
Let $\ul{E_i}=(E,h_i^E, \nabla_i^E)$, $i=0,1$, be two equivariant 
versions of geometrc triples over $W$. 
Let $D_{X,0}^E$ and $D_{X,1}^E$ be two fiberwise Dirac operators
associated with $(g_0^{TX}, \nabla_0^{E})$ and $(g_1^{TX}, \nabla_1^{E})$ respectively. Assume that $\ker(D_{X,0}^E)$ and $\ker(D_{X,1}^E)$ form vector bundles over $B$.
Let $\tilde{\bar{\eta}}^{BC}_{g,0}(\pi,E)$ and $\tilde{\bar{\eta}}^{BC}_{g,1}(\pi,E)$ be corresponding reduced equivariant eta forms in (\ref{eq:0.10}). 
Let $K_G^*(B)$ be the equivariant topological $K$-group of $B$.
Let $\ch_g:K_G^*(B)\to H^*(B,\C)$ be the equivariant Chern 
character map (see (\ref{eq:1.07}) for the definition of $\ch_g$ on 
$K_G^1(B)$) for $g\in G$.

\begin{thm}\label{thm:0.04} 
	There exists $x\in K_G^*(B)$, such that modulo exact forms on $B$, for $g\in G$,
	\begin{multline}\label{eq:0.09}
\tilde{\bar{\eta}}^{BC}_{g,1}(\pi,E)-
	\tilde{\bar{\eta}}^{BC}_{g,0}(\pi,E)
	=\int_{X^g}\wi{\Ahat}_g\left(TX, \nabla_0^{TX}, \nabla_1^{TX} \right)\ch_g\left(E, \nabla_0^{E}\right)
	\\
	+\int_{X^g}\Ahat_g\left(TX, \nabla_1^{TX}\right)\wi{\ch}_g\left(E,\nabla_0^{E},\nabla_1^{E}\right)
	+\ch_g(x).
	\end{multline}
 Here  $\wi{\Ahat}_g(\cdot)$, $\wi{\ch}_g(\cdot)$ are equivariant
 Chern-Simons forms associated with equivariant $\Ahat$-forms, 
 equivariant Chern
 character forms respectively, which are the 
 natural equivariant versions of \cite[Definition B.5.3]{MM07}.
\end{thm} 

Remark that if $B$ is a point, $\dim X$ is odd and $G=\{e\}$, Theorem \ref{thm:0.04}
degenerates to Theorem \ref{thm:0.01}.
If $G=\{ e\}$, Theorem \ref{thm:0.04} follows directly
from the family APS index theorem \cite{MP97a,MP97b}
for the submersion $W\times [0,1]\to B$ with fiber $X\times [0,1]$.
In this case, the element $x\in K^*(B)$ in (\ref{eq:0.09}) 
is a family APS index. 
For the equivariant case,
 in \cite{Liu17}, the author established this formula
under the assumption that there exists a smooth path connecting
the geometric data $(T_0^HW, g_0^{TX}, h_0^E, \nabla_0^E)$ and $(T_1^HW, g_1^{TX}, h_1^E, \nabla_1^E)$ such that there is no 
higher spectral flow along this path. 
Here we remove this additional assumption.
In \cite[Theorem 1.2]{Liu21}, the author established the variation 
formula for the equivariant eta forms with perturbation
(see Proposition \ref{prop:2.01}). 
For variation formula of eta forms for other settings,
see e.g., \cite{BM04,Bu09,Lo94,GL18}.

\subsection{Embedding formula}\label{s0004} 
Let $i:W\rightarrow V$ be an equivariant embedding of
two closed $G$-manifolds with even codimension. Let 
$\pi_V:V\rightarrow B$ be an equivariant submersion with closed
fiber $X$, whose restriction $\pi_W:W\rightarrow B$ is also an equivariant 
submersion with closed fiber $Y$. We assume that $G$ acts on $B$ 
trivially. We have the diagram of fibrations:

\begin{center}
	\begin{tikzpicture}[>=angle 90]
	\matrix(a)[matrix of math nodes,
	row sep=2em, column sep=2.5em,
	text height=1.5ex, text depth=0.25ex]
	{&Y&W\\
		&X&V&B.\\};
	\path[->](a-1-2) edge node[left]{\footnotesize{$i$}} (a-2-2);
	\path[->](a-1-2) edge (a-1-3);
	\path[->](a-2-2) edge (a-2-3);
	\path[->](a-1-3) edge node[left]{\footnotesize{$i$}} (a-2-3);
	\path[->](a-2-3) edge node[above]{\footnotesize{$\pi_V$}} (a-2-4);
	\path[->](a-1-3) edge node[above]{\footnotesize{$\pi_W$}} (a-2-4);
	\end{tikzpicture}
	\end{center}

Let $T^HV$ be an equivariant horizontal subbundle
over $V$. 
Assume that $T^HV|_W\subset TW$.
Set $T^HW:=T^HV|_W$. Then $T^HW$ is an equivariant horizontal
subbundle over $W$.
 Let $g^{TX}$ 
be an equivariant metric on $TX$ and $g^{TY}$ be the induced metric
on $TY$.
We assume that $TY$ and $TX$ have equivariant spin structures.
Let $N$ be the normal bundle of $TY$ in $TX$. Let $\nabla^{TY}$
and
$\nabla^{TX}$ be connections in (\ref{eq:1.01}) associated with
$(T^HW, g^{TY})$ and $(T^HV, g^{TX})$ respectively. Let $\nabla^N$ be the 
connection on $N$ induced from $\nabla^{TX}$.
Set $\nabla^{TY,N}:=\nabla^{TY}\oplus \nabla^{N}$.

Let $\ul{\mu}=(\mu, h^{\mu}, \nabla^{\mu})$ be an equivariant
 geometric triple over 
$W$. As in Theorem \ref{thm:0.02}, by \cite[\S 3.3]{Embed}
(see also \cite[(1.10)]{BZ93}, \cite[\S 1.4]{LiuMa20}), we can construct the 
equivariant version of the Atiyah-Hirzebruch direct image: the  equivariant geometric triples
$\ul{\xi_{\pm}}=(\xi_{\pm}, h^{\xi_{\pm}},\nabla^{\xi_{\pm}})$
over $V$ satisfying the equivariant metric condition \cite[(3.13)]{Embed}. Assume that $\ker D_X^{\xi_+}$, $\ker D_X^{\xi_-}$
and $\ker D_Y^{\mu}$ form vector bundles over $B$.

\begin{thm}\label{thm:0.05}\footnote{For Clifford module
		case, Theorem \ref{thm:0.05} needs an additional assumption 
		that $N$ has an equivariant 
		spin$^c$ structure as in \cite{Embed}.}
There exists $x\in K_G^*(B)$ such that modulo exact forms on $B$,
for $g\in G$,
	\begin{multline}\label{eq:0.11}
	\tilde{\bar{\eta}}^{BC}_g(\pi_V,\xi_+)-\tilde{\bar{\eta}}^{BC}_g(\pi_V, \xi_-)
	=\tilde{\bar{\eta}}^{BC}_g(\pi_W,\mu)+\int_{X^g}\Ahat_g(TX, \nabla^{TX})\gamma_g(Y^g,X^g)
	\\
	+\int_{Y^g}\wi{\Ahat}_g(TY, \nabla^{TY,N},\nabla^{TX|_{W^g}})
	\Ahat^{-1}_g(N,\nabla^N)\ch_g(E,\nabla^E)+\ch_g(x).
	\end{multline}
	Here $\gamma_g(Y^g,X^g)$ is the equivariant Bismut-Zhang current defined in \cite{Embed}.
\end{thm}

When $B$ is a point, $\dim X$ is odd and $G=\{e\}$, Theorem \ref{thm:0.05}
degenerates to Theorem \ref{thm:0.02}. Remark that in \cite{Embed}, the author obtained the equivariant version of Theorem \ref{thm:0.02} and generalized it to the equivariant eta forms with perturbation.

\subsection{Functoriality}\label{s0005}
Let $W$, $V$, $B$ be closed $G$-manifolds. Let $\pi_1: W\rightarrow V$, $\pi_2: V\rightarrow B$ be equivariant submersions with closed 
fibers $X$, $Y$. Then $\pi_3=\pi_2\circ \pi_1: W\rightarrow B$ is an equivariant submersion with closed fiber $Z$.
Assume that $G$ acts on $B$ trivially.
We have the diagram of fibrations:

\begin{center}\label{i3}
	\begin{tikzpicture}[>=angle 90]
	\matrix(a)[matrix of math nodes,
	row sep=2em, column sep=2.5em,
	text height=1.5ex, text depth=0.25ex]
	{X&Z&W\\
		&Y&V&B.\\};
	\path[->](a-1-1) edge (a-1-2);
	\path[->](a-1-2) edge node[left]{$\pi_1$} (a-2-2);
	\path[->](a-1-2) edge (a-1-3);
	\path[->](a-2-2) edge (a-2-3);
	\path[->](a-1-3) edge node[left]{$\footnotesize{\pi_1}$} (a-2-3);
	\path[->](a-2-3) edge node[above]{$\footnotesize{\pi_2\ \ }$} (a-2-4);
	\path[->](a-1-3) edge node[above]{$\footnotesize{\ \pi_3}$} (a-2-4);
	\end{tikzpicture}
\end{center}

Assume that relative tangent bundles $TX$ and $TY$ have equivariant
spin structures. So is the relative tangent bundle $TZ\simeq
\pi_1^*TY\oplus TX$. Let $(T_1^HW, g^{TX})$, $(T_2^HV, g^{TY})$
and $(T_3^HW, g^{TZ})$ be equivariant geometric data with respect to
$\pi_1$, $\pi_2$ and $\pi_3$ as in Section \ref{s0002}.
Let $\nabla^{TX}$, $\nabla^{TY}$ and $\nabla^{TZ}$ be the 
corresponding connections on $TX$, $TY$ and $TZ$ as in 
(\ref{eq:1.01}). Set $\nabla^{TY,TX}:=\pi_1^*\nabla^{TY}
\oplus \nabla^{TX}$.
Let $\ul{E}=(E,h^E, \nabla^E)$ be an equivariant geometric triple over $W$. 

Let $D_X^E$ and $D_Z^E$ be fiberwise Dirac operators
associated with $(g^{TX}, \nabla^E)$ and $(g^{TZ}, \nabla^E)$. Assume that $\ker D_X^E$ (resp. $\ker D_Z^E$) forms 
a vector bundle over $V$ (resp. $B$). Let $\nabla^{\ker}$ be the 
induced $G$-invariant connection on the vector bundle $\ker D_X^E$
as in (\ref{eq:1.03}).
Let $D_Y^{\ker D_X^E}$ be the fiberwise Dirac operator twisted with 
the vector bundle $\ker D_X^E$ over $V$ associated with
$(g^{TY}, \nabla^{\ker})$ such that $\ker D_Y^{\ker D_X^E}$ forms 
	a vector bundle over $B$.

\begin{thm}\label{thm:0.06}
	There exists $x\in K_G^*(B)$, such that
	modulo exact forms on $B$, for $g\in G$, if $\dim X$ is even,
	\begin{multline}\label{eq:0.12}
	\tilde{\bar{\eta}}^{BC}_g(\pi_3, E)=
	\tilde{\bar{\eta}}^{BC}_g(\pi_2, \ker D_X^E)
	+\int_{Z^g}\widetilde{\widehat{\mathrm{A}}}_g(TZ,\nabla^{TY,TX},\nabla^{TZ}  )
	\, \ch_g(E, \nabla^{E})
	\\
	+\int_{Y^g}\widehat{\mathrm{A}}_g(TY, \nabla^{TY})\, \tilde{\bar{\eta}}^{BC}_g(\pi_1, E)+\ch_g(x);
	\end{multline}
if $\dim X$ is odd,
	\begin{multline}\label{eq:0.14}
\tilde{\bar{\eta}}^{BC}_g(\pi_3, E)=
\int_{Z^g}\widetilde{\widehat{\mathrm{A}}}_g(TZ,\nabla^{TY,TX},\nabla^{TZ}  )
\, \ch_g(E, \nabla^{E})
\\
+\int_{Y^g}\widehat{\mathrm{A}}_g(TY, \nabla^{TY})\, \tilde{\bar{\eta}}^{BC}_g(\pi_1, E)+\ch_g(x).
\end{multline}
\end{thm}

Remark that when $B$ is a point, $G=\{e\}$, Theorem \ref{thm:0.06}
is a generalization of Theorem \ref{thm:0.03}. In fact, let 
$\nabla_t^{TZ}$ be the Levi-Civita connection associated with
$g^{TX}\oplus t^{-2}\pi_1^*g^{TY}$. Then by \cite[Proposition 4.5]{Liu17}
(cf. also \cite[(4.32)]{Ma99}), $\displaystyle\lim_{t\to 0}\widetilde{\widehat{\mathrm{A}}}_g(TZ,\nabla^{TY,TX},\nabla_t^{TZ}  )=0$. The hidden 
mod $\Z$ term in 
(\ref{eq:0.05}) is explained here as a spectral flow. 

In \cite{Liu17}, we established this formula under the assumptions that $\ker D_Y^{\ker D_X^E}=0$ and there is no higher spectral flow in all deformations. Here we remove these assumptions. 
Notice that if $\dim X$ is odd, (\ref{eq:0.14}) seems different
from \cite[Theorem 1.3]{Liu17}. But they are the same under the assumptions in \cite[Theorem 1.3]{Liu17}. The reason is that 
if $\dim X$ is odd, the term $\tilde{\eta}_g(T_2^HV,g^{TY},h^{L_Y},
h^{\ker D^X},\nabla^{L_Y},\nabla^{\ker D^X})$ in \cite[(1.6)]{Liu17} vanishes by counting 
degrees on both sides.
In \cite{Liu21}, we obtain the functoriality for equivariant eta forms with perturbation. See also \cite{BB94,Ma99,Ma00,BM04,BS09,BS13,GL18} for other settings. 

Note that when the fiberwise Dirac operator $D$ is a fibrewise 
signature operator, $\ker D$ naturally forms a vector bundle.
In this case, the mod $\Z$ term in (\ref{eq:0.05}) was 
constructed by spectral sequences in \cite{Dai91}.
When the formulas are extended to the equivariant family 
case, the spectral sequence terms should be generalized
to the equivariant Chern characters of some vector bundles
arising from the study of spectral sequence as in \cite{Ma00,BM04}.

Naturally, the functoriality can be extended to the multifibration
case in \cite[Appendix 2]{BC89}.

\subsection{The organization of the article}\label{s0006}

This article  is organized as follows. In Section \ref{s02},
we recall the definition of the equivariant Bismut-Cheeger eta form and
compare it with the
equivariant eta form with perturbation. In Sections \ref{s03}-\ref{s05},
we prove Theorems \ref{thm:0.04}-\ref{thm:0.06} using the 
comparison formulas in Section \ref{s02}.

\

\textbf{Notation}.
All manifolds in this paper are smooth and without boundary.
We denote by $d$ the exterior differential operator and 
$d^B$ when we like to insist the base manifold $B$.

We use the superconnection formalism of Quillen \cite{Qu85}.
If $A$ is a $\Z_2$-graded algebra, and if $a,b\in A$, 
then we will note $[a,b]:=ab-(-1)^{\deg(a)\deg(b)}ba$ as the supercommutator 
of $a, b$.
If $E, E'$ are two $\Z_2$-graded spaces,
we will note $E\widehat{\otimes}E'$ as the $\Z_2$-graded 
tensor product as in \cite[\S 1.3]{BGV04}.
If one of $E, E'$ is ungraded, we understand 
it as $\Z_2$-graded by taking its odd part as zero. 

For the fiber bundle $\pi: W\rightarrow B$, 
we use the sign convention for the integration of  the differential forms along the oriented fibers $Z$ as follows:
for $\alpha\in \Omega^{\bullet}(B)$ and 
$\beta\in \Omega^{\bullet}(W)$, 
\begin{align}\label{e01136}
\int_Z(\pi^*\alpha)\wedge\beta=\alpha\wedge \int_Z\beta.
\end{align}


\settocdepth{subsection}
\section{Bismut-Cheeger Eta forms}\label{s02}

In this section, we recall the definitions of the equivariant
Bismut-Cheeger eta form and the equivariant version of the
 eta form with 
perturbation, which was originally introduced in \cite{MP97a,MP97b}, 
and  study the relations between them.




Let $G$ be a compact Lie group.
Let $\pi:W\to B$ be an equivariant submersion of two compact 
$G$-manifolds with compact fiber $X$ such that 
$G$ acts on $B$ trivially.

Let $TX:=\ker(\pi_*:TW\to TB)$ be the relative tangent bundle.
Then $TW$ and $TX$ are equivariant vector bundles over $W$.
Let $T^HW\subset TW$ be an 
equivariant horizontal subbundle, such that
\begin{align}\label{eq:1.00}
TW=T^HW\oplus TX.
\end{align}
Since $G$ is compact, such $T^HW$ always exists. Let $P^{TX}:TW
\rightarrow TX$ be the projection associated with (\ref{eq:1.00}). Note that
$
T^HW\cong \pi^*TB.
$

Let $g^{TX}$, $g^{TB}$ be $G$-invariant  metrics on $TX$,
$TB$. 
We equip $TW=T^HW\oplus TX$ with the 
$G$-invariant metric $
g^{TW}=\pi^*g^{TB}\oplus g^{TX}.
$
Let $\nabla^{TW}$ be the Levi-Civita 
connection on $(TW,g^{TW})$. 
Let $\nabla^{TX}$ be the connection on $TX$ defined by
\begin{align}\label{eq:1.01}
\nabla^{TX}=P^{TX}\nabla^{TW}P^{TX}.
\end{align}
It is a $G$-invariant Euclidean connection on $TX$ which depends only on 
$(T^HW, g^{TX})$ (cf. \cite[Theorem 1.9]{Bi86}).
Let $\nabla^{TB}$ be the Levi-Civita connection on $(TB, g^{TB})$.
Let $\nabla^{TB,TX}$ be the connection on $TW$ defined by
\begin{align}\label{eq:1.11}
\nabla^{TB,TX}=\pi^*\nabla^{TB}\oplus \nabla^{TX},
\end{align}
which is also $G$-invariant.

Let $C(TX)$ be the Clifford algebra bundle of $(TX, g^{TX})$, whose fiber at $w\in W$ is the Clifford algebra
$C(T_wX)$ of the Euclidean space $(T_wX, g^{T_wX})$.
We assume that $TX$ has a $G$-equivariant spin structure.
Let $\mathcal{S}_X$ be the spinor bundle over $W$ for 
$(TX, g^{TX})$, which has a smooth action of $C(TX)$ and is 
$G$-equivariant. Let $\nabla^{\mathcal{S}_X}$ be the
$G$-invariant connection
on $\mathcal{S}_X$ induced by $\nabla^{TX}$.
If $\dim X$ is even, $\mathcal{S}_X$ is naturally $\Z_2$-graded
and $\nabla^{\mathcal{S}_X}$ preserves this $\Z_2$-grading.

Let $E$ be an equivariant complex vector bundle over $W$.
Let $h^{E}$ be a $G$-invariant Hermitian metric on $E$. 
Let $\nabla^{E}$ be a $G$-invariant Hermitian connection 
on $(E,h^E)$. 
As in Section \ref{s0001}, we say $\ul{E}=(E,h^E,\nabla^E)$
is an equivariant geometric triple over $W$.
Set
$
\nabla^{\mathcal{S}_X\otimes E}:=\nabla^{\mathcal{S}_X}\otimes 1+1\otimes \nabla^E.
$
Then $\nabla^{\mathcal{S}_X\otimes E}$ is a $G$-invariant Hermitian connection on $(\mathcal{S}_X\otimes E, h^{\mathcal{S}_X}\otimes h^E)$.

Let $\{e_i \}_{i=1}^{\dim X}$ be a local orthonormal frame of $TX$.
The fiberwise Dirac operator is defined by
\begin{align}\label{eq:1.02}
D_X^E=\sum_{i=1}^{\dim X}c(e_{i})\nabla^{\mathcal{S}_X\otimes E}_{e_i},
\end{align}
which is independent of the choice of the local orthonormal frame.
If $\dim X$ is even, the fiberwise Dirac operator
$D_X^E=D_+\oplus D_-$ with respect to the $\Z_2$-grading.

For $b\in B$, let $\mathcal{E}_{b}$ be the set of smooth 
sections over $X_b=\pi^{-1}(b)$ of 
$\mathcal{S}_X\otimes E|_{X_b}$. As in \cite{Bi86},
we will regard $\mathcal{E}$ as
an infinite dimensional vector bundle over $B$.
If $\dim X$ is even, then $\mathcal{E}$ is $\Z_2$-graded.

Let $\nabla^{\mathcal{E}}$ be the connection on $\mathcal{E}$
defined in \cite[(1.7)]{BF86}, which preserves the $L^2$ inner
product on $\mathcal{E}$.
If $U\in TB$, let $U^H\in T^HW$ be its horizontal lift in $T^HW$ so that $\pi_*U^H=U$.
Let $\{f_p \}$ be a local orthonormal frame of $TB$.
We denote by
$
c(T)=-\frac{1}{2}\,c\left(P^{TX}[f_p^H, f_q^H]\right)f^p\wedge f^q\wedge.
$
Let $\mathbb{B}_u$ be the rescaled Bismut superconnection defined by (see e.g., \cite[p336]{BGV04})
\begin{align}\label{eq:1.12}
\mathbb{B}_u=\sqrt{u}D_X^E+\nabla^{\mathcal{E}}-\frac{c(T)}{4\sqrt{u}},
\quad u>0.
\end{align}
Note that $\mathbb{B}_u^2$ is a 2nd-order 
elliptic differential operator along the fibers \cite[Theorem 3.5]{Bi86}.

We assume that $\ker D_X^E$ forms a vector bundle over $B$. Then the $L^2$ inner product on $\mathcal{E}$ induces 
a $G$-invariant metric on $\ker D_X^E$.
Let $P^{\ker}:\mathcal{E}\to \ker D_X^E$ be the orthogonal 
projection with respect to the $L^2$ inner product on $\mathcal{E}$. Let
\begin{align}\label{eq:1.03}
\nabla^{\ker}:=P^{\ker}\circ \nabla^{\mathcal{E}}\circ P^{\ker}.
\end{align}
It is a $G$-invariant Hermitian connection on $\ker D_X^E$.

We define the supertrace $\tr_s$ on a trace class operator
on $\mathrm{End}(\mathcal{E})$ as in \cite[\S 1.3]{BGV04}.
If $P$ is a trace class operator acting on $\Lambda(T^*B)\widehat{\otimes}\mathrm{End}(\mathcal{E})$ which takes values in $\Lambda(T^*B)$,
we use the convention that if $\omega\in \Lambda(T^*B)$,
\begin{align}\label{e01056}
\tr[\omega P]=\omega\tr[P],\quad
\tr_s[\omega P]=\omega\tr_s[P].
\end{align}
We denote by $\tr^{\mathrm{odd/even}}[P]$ the part of $\tr[P]$
which takes values in odd or even forms. Set
\begin{align}\label{i16}
\widetilde{\tr}[P]=
\left\{
\begin{array}{ll}
\tr_s[P], & \hbox{if $\dim X$ is even;} \\
\tr^{\mathrm{odd}}[P], & \hbox{if $\dim X$ is odd.}
\end{array}
\right.
\end{align}

For $\alpha\in \Omega^j(B)$, set
\begin{align}\label{1.04}
\psi_B(\alpha)=\left\{
\begin{array}{ll}
\left(\frac{1}{2\pi \sqrt{-1}}\right)^{\frac{j}{2}}\cdot \alpha, & \hbox{if $j$ is even;} \\
\frac{1}{\sqrt{\pi}}\left(\frac{1}{2\pi \sqrt{-1}}\right)^{\frac{j-1}{2}}\cdot \alpha, & \hbox{if $j$ is odd.}
\end{array}
\right.
\end{align}

For $\beta\in \Omega^*(B\times [0,1]_u)$,
if
$\beta=\beta_0+du\wedge \beta_1$, with $\beta_0, \beta_1\in 
\Lambda (T^*B)$, set
\begin{align}
[\beta]^{du}:=\beta_1.
\end{align}

The following definition is the equivariant version of the 
Bismut-Cheeger eta form in \cite{BC89}.

\begin{defn}\label{defn:1.01} \cite[Definition 2.3]{Liu17}
For $g\in G$, the equivariant Bismut-Cheeger eta form is defined by
\begin{multline}\label{eq:1.05}
\tilde{\eta}^{BC}_g(\pi,E):=
-\int_0^{\infty}\left\{\psi_{\R\times B}\left.\widetilde{\tr}\right.\left[g\exp\left(-\left(
\mathbb{B}_{u}+du\wedge\frac{\partial}{\partial u}\right)^2\right)\right]\right\}^{du}du
\\
=
\left\{
\begin{aligned}
& \int_0^{\infty}\left.\frac{1}{\sqrt{\pi}}\psi_{B}\tr^{\mathrm{even}}\right.\left[g\left.\frac{\partial \mathbb{B}_u}{\partial u}\right.
\exp(-\mathbb{B}_u^{2})\right] du \in \Omega^{\mathrm{even}}(B,\C),
\\
& \quad\quad\quad\quad\quad\quad\quad\quad\quad\quad\quad\quad\quad\quad\quad\quad\quad\quad\quad
\quad\hbox{if $\dim X$ is odd;} \\
&\int_0^{\infty} \left.\frac{1}{2\sqrt{\pi}\sqrt{-1}}\psi_{B}\tr_s\right.\left[g\left.\frac{\partial \mathbb{B}_u}{\partial u}\right.
\exp(-\mathbb{B}_u^{2})\right] du
\in \Omega^{\mathrm{odd}}(B,\C),
\\
& \quad\quad\quad\quad\quad\quad\quad\quad\quad\quad\quad\quad\quad\quad\quad\quad\quad
\quad\quad\quad\hbox{if $\dim X$ is even.} \\
\end{aligned}
\right.
\end{multline}	
\end{defn}

Note that the convergence of the integrals in (\ref{eq:1.05}) are highly nontrivial \cite[(2.72), (2.77)]{Liu17} (cf. also \cite[Theorem 10.32]{BGV04}).
Remark that by changing the variable (see also \cite[Remark 2.4]{Liu17}), 
\begin{align}\label{i18}
\tilde{\eta}^{BC}_g(\pi,E)
=
-\int_0^{\infty}\left\{\psi_{\R \times B}\left.\widetilde{\tr}\right.\left[g\exp\left(-\left(
\mathbb{B}_{u^2}+du\wedge\frac{\partial}{\partial u}\right)^2\right)\right]\right\}^{du}du.
\end{align}

We define the reduced equivariant eta form 
$\tilde{\bar{\eta}}^{BC}_g(\pi, E)$ as in (\ref{eq:0.10}).
If $E=E_+\oplus E_-$ is $\Z_2$-graded, we define
$\tilde{\bar{\eta}}^{BC}_g(\pi, E)=\tilde{\bar{\eta}}^{BC}_g(\pi, E_+)-\tilde{\bar{\eta}}^{BC}_g(\pi, E_-)$.

\

Now we recall the definition of the equivariant eta form with perturbation in \cite{Liu21}.

Let $K_G^0(B)$ be the equivariant topological $K^0$-group of $B$. Fix $s\in S^1$. Let $\iota:B\rightarrow B\times S^1$,
$b\mapsto (b,s)$, be the $G$-equivariant 
inclusion map, where we suppose that the $G$-action on $S^1$ is trivial. From
\cite[Definitions 2.7, 2.8]{Se68},
\begin{align}\label{eq:1.06}
K_G^1(B)\simeq \ker\left(\iota^*:K_G^0(B\times S^1)\rightarrow 
K_G^0(B)\right).
\end{align}
Recall that $G$ acts on $B$ trivially.
For $x\in K_G^0(B)$, $g\in G$, the classical equivariant Chern character map 
sends $x$ to $\ch_g(x)\in H^{\mathrm{even}}(B,\C)$.
By (\ref{eq:1.06}), for $x\in K_G^1(B)$, we can regard $x$ as an element $x'$ 
in $K_G^0(B\times S^1)$. The odd equivariant Chern character map
\begin{align}\label{eq:1.07}
\ch_g:K_G^1(B)\longrightarrow H^{\mathrm{odd}}(B,\C)
\end{align}
is defined by
$
\ch_g(x):=\int_{S^1}\ch_g(x').$

If $\dim X$ is even (resp. odd), the equivariant (analytic)
index $\ind(D_X^E)\in K_G^0(B)$ (resp. $K_G^1(B)$).
If $\ind(D_X^E)=0\in K_G^*(B)$, by \cite[Proposition 3.3 (i)]{Liu21}
(cf. also \cite[Proposition 1]{MP97a} and \cite[Proposition 2]{MP97b}),
there exists a smooth family of equivariant
bounded pseudodifferential operators $\mA$ such that $(D_X^E+\mA)|_{X_b}$ is invertible 
for any $b\in B$. If $\dim X$ is even, $\mA$ is additionally
required to anti-commute with the $\Z_2$-grading of the 
spinor $\mathcal{S}_X$\footnote{If $E=E_+\oplus E_-$ is
	 $\Z_2$-graded, $\mA$ is required to anti-commute with 
	 the $\Z_2$-grading of $\mathcal{S}_X\widehat{\otimes}E$
	when $\dim X$ is even and commute with the $\Z_2$-grading
	of $\mathcal{S}_X\widehat{\otimes}E$ when $\dim X$
	is odd.}. Such operator $\mA$ is called a perturbation operator.

Let $\chi\in \cC^{\infty}(\R)$ be a cut-off function such that
\begin{align}\label{eq:1.08}
\chi(u)=
\left\{
\begin{aligned}
&0,  &\hbox{if $u<1$;} \\
&1,  &\hbox{if $u>2$.}
\end{aligned}
\right.
\end{align}
Set
\begin{align}\label{eq:1.09}
\mathbb{B}_u'=\mathbb{B}_u+\sqrt{u}\chi(\sqrt{u})\mA.
\end{align}
The following definition is the equivariant version of the eta
form with perturbation in \cite{MP97a,MP97b,DZ98}.

\begin{defn}\label{defn:1.02} \cite[Definition 3.12]{Liu21}
Modulo exact forms on $B$, the equivariant eta form with perturbation operator $\mA$ is defined by
\begin{multline}\label{eq:1.10}
\tilde{\eta}_g(\pi,\mA):=
-\int_0^{\infty}\left\{\psi_{\R\times B}\left.\widetilde{\tr}\right.\left[g\exp\left(-\left(
\mathbb{B}_{u}'+du\wedge\frac{\partial}{\partial u}\right)^2\right)\right]\right\}^{du}du
\\
=
\left\{
\begin{aligned}
& \int_0^{\infty}\left.\frac{1}{\sqrt{\pi}}\psi_{B}\tr^{\mathrm{even}}\right.\left[g\left.\frac{\partial \mathbb{B}_u'}{\partial u}\right.
\exp(-(\mathbb{B}_u')^{2})\right] du,
\quad\hbox{if $\dim X$ is odd;} \\
&\int_0^{\infty} \left.\frac{1}{2\sqrt{\pi}\sqrt{-1}}\psi_{B}\tr_s\right.\left[g\left.\frac{\partial \mathbb{B}_u'}{\partial u}\right.
\exp(-(\mathbb{B}_u')^{2})\right] du,
\quad\hbox{if $\dim X$ is even.} \\
\end{aligned}
\right.
\end{multline}	 
\end{defn}

Moreover, as in (\ref{eq:0.08}) (see e.g., \cite[(3.66)]{Liu21}),
\begin{align}\label{eq:1.24}
d\tilde{\eta}_g(\pi,\mA)=\int_{X^g}\Ahat_g(TX,\nabla^{TX})\ch_g(E,\nabla^E).
\end{align}
As in (\ref{i18}), we have
\begin{align}\label{eq:1.25}
\tilde{\eta}_g(\pi,\mA)
=
-\int_0^{\infty}\left\{\psi_{\R \times B}\left.\widetilde{\tr}\right.\left[g\exp\left(-\left(
\mathbb{B}_{u^2}'+du\wedge\frac{\partial}{\partial u}\right)^2\right)\right]\right\}^{du}du.
\end{align}

The proof of the well-definedness of 
$\tilde{\eta}_g(\pi,\mA)$ is the same as 
that of $\tilde{\eta}_g^{BC}(\pi,E)$. 
Remark that if we choose another cut-off function, the difference
of the new eta form and the original one is an exact form.
So modulo exact forms, the definition of $\tilde{\eta}_g(\pi,\mA)$ is independent of the choice of the cut-off function in 
(\ref{eq:1.08}).

\

Now we discuss the relations between these two types of 
equivariant eta forms when $\ker D_X^E$ forms a
vector bundle over $B$.

\

\textbf{Case 1}: $\dim X$ is odd.

\

Since $\ker D_X^E$ forms a vector bundle,
 $\ker (D_X^E+P^{\ker})=0$. So from \cite[Proposition 3.3]{Liu21}
 (cf. also \cite[Proposition 1]{MP97a}),
we have $\ind (D_X^E)=0\in K_G^1(B)$
and $P^{\ker }$ is a perturbation operator.

\begin{prop}\label{prop:1.03}
	If $\dim X$ is odd, we have
	\begin{align}\label{eq:1.13}
	\tilde{\eta}_g(\pi,P^{\ker})=\tilde{\bar{\eta}}^{BC}_g(\pi, E) \mod d\Omega^*(B).
	\end{align}
\end{prop}
\begin{proof}
	Let $\widetilde{\mathbb{B}}_{u}$ be the Bismut superconnection with respect to the fiber bundle
	$W\times [0,1]_s\to B\times[0,1]_s$ with fiber $X$
	such that restricted on $W\times \{s\}$, $s\in [0,1]$,
	$\widetilde{\mathbb{B}}_{u}|_{W\times \{s\}}=\mathbb{B}_u+ds\wedge \frac{\partial}{\partial s}$.
	Set
	\begin{align}\label{eq:1.14}
	\widetilde{\mathbb{B}}_{u^2}':=\widetilde{\mathbb{B}}_{u^2}
	+us\chi(u)P^{\ker}=u(D_X^E+s\chi(u)P^{\ker})+\nabla^{\mathcal{E}}+
	ds\wedge\frac{\partial}{\partial s}-\frac{c(T)}{4u}.
	\end{align}
	
	We decompose
	\begin{multline}\label{e01112}
	\psi_{\mathbb{R}^2\times B}\tr^{\mathrm{odd}}\left[g\exp\left(-\left(\widetilde{\mathbb{B}}_{u^2}'+du\wedge\frac{\partial}{\partial u}\right)^2\right)\right]
	\\
	=du\wedge\gamma(u,s)+ds\wedge r_1(u,s)+du\wedge ds\wedge r_2(u,s)+r_3(u,s),
	\end{multline}
	where $\gamma, r_1, r_2, r_3$ do not contain $du$ neither $ds$.
	From (\ref{i18}), (\ref{eq:1.25}) and (\ref{e01112}), we have
	\begin{align}\label{e01203}
	\tilde{\eta}_g^{BC}(\pi,E)=-\int_0^\infty \gamma(u,0)du, \quad
	\tilde{\eta}_g(\pi,P^{\ker})=-\int_0^\infty \gamma(u,1)du.
	\end{align}
	
	By \cite[Theorem 9.17]{BGV04}, 
	\begin{align}\label{eq:2.24}
	\left(du\wedge \frac{\partial}{\partial u}+ds\wedge \frac{\partial}{\partial s}+d^{B}\right)
	\tr\left[g\exp\left(-\left(\widetilde{\mathbb{B}}_{u^2}'+du\wedge\frac{\partial}{\partial u}\right)^2\right)\right]=0.
	\end{align}
	So from (\ref{e01112}) and (\ref{eq:2.24}),
	\begin{align}\label{e01113}
	\frac{\partial \gamma(u,s)}{\partial s}=\frac{\partial r_1(u,s)}{\partial u}+d^B r_2(u,s).
	\end{align}
	From (\ref{e01203}) and (\ref{e01113}),
	\begin{multline}\label{e01114}
	\tilde{\eta}_g(\pi,P^{\ker})-\tilde{\eta}_g^{BC}(\pi,E)
	=-\int_{0}^{+\infty}(\gamma(u,1)-\gamma(u,0))du
	\\
	=-\int_0^{+\infty}\int_0^1 \frac{\partial}{\partial s}\gamma(u,s)ds du=-\int_0^1\int_0^{+\infty} \frac{\partial}{\partial s}\gamma(u,s)du ds
	\\
	=-\int_0^1\int_0^{+\infty} \frac{\partial}{\partial u}r_1(u,s)du ds-d^B\int_0^1\int_0^{+\infty} r_2(u,s)du ds
	\\
	=\int_0^1 \left(r_1(0,s)-r_1(\infty,s)\right)ds-d^B\int_0^1\int_0^{+\infty} r_2(u,s)du ds.
	\end{multline}
	The commutative property of the integrals
	in the above formula is guaranteed by  \cite[(2.72), (2.77)]{Liu17} for $s\in [0,1]$.
	By (\ref{eq:1.08}), (\ref{e01112}) and the equivariant family local index theorem (see e.g., \cite[Theorem 1]{LiuMa00}, \cite[Theorem 2.2]{Liu17}),
	as in \cite[(2.96), (2.97)]{Liu17}, we have 
	\begin{align}\label{eq:1.26}
	r_1(0,s)=0,\quad r_1(\infty,s)=\frac{1}{\sqrt{\pi}}
	\lim_{u\to+\infty}\psi_B\left\{\tr^{\mathrm{odd}}
	\left[g\exp\left(-\left(	\widetilde{\mathbb{B}}_{u^2}'\right)^2 \right) 
	\right]\right\}^{ds}.
	\end{align}
Therefore, by (\ref{e01114}) and (\ref{eq:1.26}), modulo exact forms on $B$, we have
	\begin{align}\label{eq:1.15}
	\tilde{\eta}_g(\pi,P^{\ker})-\tilde{\eta}^{BC}_g(\pi,E)=-
	\frac{1}{\sqrt{\pi}}\lim_{u\to+\infty}
	\psi_B\int_0^1\left\{\tr^{\mathrm{odd}}\left[g\exp\left(-\left(	\widetilde{\mathbb{B}}_{u^2}'\right)^2 \right) \right]\right\}^{ds}ds.
	\end{align}
	
		For a family of bounded operators $\mA_u$, $u\in \R_+$, we write
	$\mA_u=O(u^{-k})$ as $u\to +\infty$ if there exists $C>0$
	such that if $u$ is large enough, the norm of $\mA_u$ is dominated by $Cu^{-k}$.
	 
	Let $P^{\ker,\bot}$ be the orthogonal projection on the
	orthogonal complement of $\ker D_X^E$. Set
	\begin{align}\label{eq:1.16}
	\begin{split}
	&E_u=P^{\ker}\circ	\left(	\widetilde{\mathbb{B}}_{u^2}'\right)^2\circ P^{\ker},\quad 
	\quad F_u=P^{\ker}\circ 	\left(	\widetilde{\mathbb{B}}_{u^2}'\right)^2\circ P^{\ker,\bot},\quad 
	\\
	&G_u=P^{\ker,\bot}\circ 	\left(	\widetilde{\mathbb{B}}_{u^2}'\right)^2\circ P^{\ker},\quad 
	\, H_u=P^{\ker,\bot}\circ	\left(	\widetilde{\mathbb{B}}_{u^2}'\right)^2\circ P^{\ker,\bot}.
	\end{split}
	\end{align}		
Note that
	\begin{align}\label{eq:1.17}
	P^{\ker}\{u(D_X^E+s\chi(u)P^{\ker}) \}P^{\ker}=us\chi(u)P^{\ker}.
	\end{align}	
	Denote by $\wi{\nabla}=\nabla^{\mathcal{E}}+ds\wedge\frac{\partial}{\partial s}$. Let 
	\begin{align}\label{eq:1.18}
	\begin{split}
	&E'= u^2s^2P^{\ker} +u ds\wedge P^{\ker}  +P^{\ker}(\nabla^{\mathcal{E}})^2 P^{\ker},
	\\
	&  F'=P^{\ker,\bot} [D_X^E+sP^{\ker}, \wi{\nabla}] P^{\ker},
	\\
	& G'=P^{\ker} [D_X^E+sP^{\ker}, \wi{\nabla}] P^{\ker,\bot},
	\\ 
	& H'=P^{\ker,\bot} (D_X^E+sP^{\ker})^2 P^{\ker,\bot}=P^{\ker,\bot} (D_X^E)^2 P^{\ker,\bot}.
	\end{split}
	\end{align}
	By (\ref{eq:1.14}),
	when $u\to +\infty$,
	\begin{align}\label{eq:1.19}
	\begin{split}
	&E_u=E'+O(u^{-1}),
	\quad\quad \quad  F_u=uF'+F''+O(1),
	\\
	& G_u=uG'+G''+O(1),\quad
	H_u=u^2H'+uH''+H'''+O(1),
	\end{split}
	\end{align}
	where $F'', G'', H'', H'''$ are 1st-order differential
	operators along the fiber.
	Moreover, by (\ref{eq:1.03}) and (\ref{eq:1.18}),
	we have
	\begin{align}\label{eq:1.20}
	E'-G'H'^{-1}F'=u^2s^2P^{\ker} + u ds\wedge+(\nabla^{\ker})^2+s^2\mathcal{C}+s\mathcal{D},
	\end{align}
	where
	\begin{align}
	\begin{split}
	&\mathcal{C}=P^{\ker}\wi{\nabla}P^{\ker,\bot}(P^{\ker,\bot}
	D_X^E P^{\ker,\bot})^{-2}P^{\ker,\bot}\wi{\nabla}P^{\ker},
	\\
	&\mathcal{D}=2P^{\ker}\wi{\nabla}P^{\ker,\bot}(P^{\ker,\bot}
	D_X^E P^{\ker,\bot})^{-1}P^{\ker,\bot}\wi{\nabla}P^{\ker}.
	\end{split}
	\end{align}
	Following the same way as the proof of \cite[Theorem 5.13]{Liu17}
	 (cf. also \cite[Theorem 9.19]{BGV04}), when $u\to +\infty$, we can obtain
	\begin{align}\label{eq:1.21}
	\exp\left(-	\left(	\widetilde{\mathbb{B}}_{u^2}'\right)^2\right)
	=P^{\ker}\circ\exp\big(-(E'-G'H'^{-1}F')\big)\circ P^{\ker}+O(u^{-1}).
	\end{align}
	
	So from (\ref{eq:1.15}), as in \cite[Theorem 5.15]{Liu17}, modulo exact forms,
	\begin{multline}\label{eq:1.23}
	\tilde{\eta}_g(\pi,P^{\ker})-\tilde{\eta}^{BC}_g(\pi,E)
	=\frac{1}{\sqrt{\pi}}\lim_{u\to+\infty}\int_0^1ue^{-u^2s^2} \psi_B\tr^{\mathrm{even}}[g\exp(-(\nabla^{\ker})^2-s^2\mathcal{C}-s\mathcal{D})]ds
	\\
	=\frac{1}{\sqrt{\pi}}\lim_{u\to+\infty}\int_0^{u}e^{-t^2} \psi_B\tr^{\mathrm{even}}[g\exp(-(\nabla^{\ker})^2-u^{-2}t^2\mathcal{C}-u^{-1}t\mathcal{D})]dt.
	\end{multline}
	Using the Volterra series (cf. e.g., \cite[(2.5)]{BGV04}),
	we have
	\begin{align}\label{eq:1.36}
	\tilde{\eta}_g(\pi,P^{\ker})-\tilde{\eta}^{BC}_g(\pi,E)
	=\frac{1}{\sqrt{\pi}}\int_0^{+\infty}e^{-t^2}dt\cdot\ch_g(\ker D_X^E, \nabla^{\ker})
	=\frac{1}{2}\ch_g(\ker D_X^E, \nabla^{\ker}).
	\end{align}
	
	By (\ref{eq:0.10}) and (\ref{eq:1.36}), the proof of Proposition \ref{prop:1.03} is completed.
\end{proof} 

Remark that we can also obtain this proposition using the method in \cite[Chapter 9]{BGV04} as in \cite[Proposition 17]{MP97a}.

\begin{rem}
	(a) Observe that from (\ref{eq:0.08}) and (\ref{eq:1.24}),
	\begin{align}\label{eq:1.35}
	d^B\tilde{\eta}_g(\pi,P^{\ker})=d^B\tilde{\bar{\eta}}^{BC}_g(\pi,E)=\int_{X^g}\widehat{\mathrm{A}}_g(TX,\nabla^{TX})\ch_g(E,\nabla^E),
	\end{align}
	which is compatible with this proposition.
	
	(b) If $B$ is a point, from \cite[Remark 3.20]{Liu21}, both sides of (\ref{eq:1.13})
	are equal to the reduced eta invariant.
\end{rem}

\

\textbf{Case 2}: $\dim X$ is even.

\

In this case, $\ker D_X^E=\ker D_+\oplus \ker D_-$ is a $\Z_2$-graded vector bundle.
But $P^{\ker}$ is not a perturbation operator because $P^{\ker}$
does not anti-commute with the $\Z_2$-grading of the spinor.
In general, $\ind(D_X^E)\neq 0\in K_G^0(B)$.


We consider a new equivariant fiber bundle $\pi':W\sqcup B\to B$ with fiber $X\sqcup \{\mathrm{pt}\}$, which is the disjoint union of the 
fiber bundle $W\to B$ and a new fiber bundle $B\times \{\mathrm{pt}\}\to B$ whose fiber is a point.
We denote by $(\ker D_+)_-$ and $(\ker D_-)_+$  the vector bundles $\ker D_+$ and $\ker D_-$ over $B\times\{\mathrm{pt} \}$
with the inverse grading.
Let $D^{H}$  be the Dirac operator on $H=\mathcal{E}\oplus (\ker D_+)_-
\oplus (\ker D_-)_+$
with respect to the fiber bundle $\pi'$.
Then $D^H=D_X^E\oplus 0\oplus 0$.
Note that
\begin{align}
\ind(D^H)=\ind D_X^E-\ind D_X^E=0\in K_G^0(B).
\end{align}

Let $\tilde{\eta}_g^{BC}(\pi',E\sqcup(\ker D_X^E)^{op})$ be the equivariant Bismut-Cheeger eta form associated with 
this fiber bundle $\pi'$. 
Here $(\ker D_X^E)^{op}$ denotes the $\Z_2$-graded vector bundle
$\ker D_X^E$ with the opposite grading.
Notice that if the fiber is a point, the corresponding
Bismut-Cheeger eta form vanishes. Thus we have
\begin{align}\label{eq:1.30}
\tilde{\eta}_g^{BC}(\pi',E\sqcup(\ker D_X^E)^{op})=\tilde{\eta}_g^{BC}(\pi,E).
\end{align} 

Let $V=\ker D_+\oplus \ker D_-\oplus(\ker D_+)_-\oplus (\ker D_-)_+$. 
Let $P^V:H\to V$ be the orthgonal projection on $V$.
Let 
\begin{align}\label{eq:1.38}
\mA=\left(
\begin{array}{cc}
0&1\\
1&0
\end{array}
\right)\circ P^V
\end{align}
on $\big(\ker D_+\oplus(\ker D_-)_+\big)\oplus 
\big((\ker D_+)_-\oplus \ker D_-\big)$.
Then $\mA$ anticommutes with the $\Z_2$-grading and $\mA^2=\Id$ on $V$.
Since 
\begin{align}
(D^H+\mA)^2=D^{H,2}+\mA^2=D^{H,2}+P^V,
\end{align}
we see that $D^H+\mA$ is invertible. 
So $\mA$ is a perturation operator associated with $D^H$.
Let $\tilde{\eta}_g^H(\pi',\mA)$ be the corresponding eta form with perturbation operator $\mA$.

\begin{prop}\label{prop:1.05}
If $\dim X$ is even, we have
	\begin{align}\label{e2}
	\tilde{\eta}_g^H(\pi',\mA)=\tilde{\bar{\eta}}_g^{BC}(\pi,E) 
	\mod d\Omega^*(B).
	\end{align}
\end{prop}
\begin{proof}
	From (\ref{eq:0.10}) and (\ref{eq:1.30}), we only need to prove 
	$\tilde{\eta}_g^H(\pi',\mA)=\tilde{\eta}_g^{BC}(\pi',E\sqcup(\ker D_X^E)^{op})\mod
	 d\Omega^*(B)$.
	
	Let $\nabla^H=\nabla^{\mathcal{E}}\oplus \nabla^{\ker}$ on $H$.
	As in (\ref{eq:1.14}), we set
	\begin{align}
	\widetilde{\mathbb{B}}_{u^2}'^{H}=u(D^H+s\chi(u)\mA)+\nabla^H+
	ds\wedge\frac{\partial}{\partial s}-\frac{c(T)}{4u}
	\end{align}
	 with respect to the vector bundle
	$(W\sqcup B)\times [0,1]\to B\times[0,1]$ with fiber $X\sqcup\{\mathrm{pt}\}$.

	Then following the same process as in (\ref{e01112})-(\ref{eq:1.15}), modulo exact forms on $B$, we have
		\begin{multline}\label{eq:1.42}
	\tilde{\eta}_g(\pi',\mA)-\tilde{\eta}_g^{BC}(\pi',E\sqcup(\ker D_X^E)^{op})
	\\
	=-\frac{1}{2\sqrt{\pi}\sqrt{-1}}\lim_{u\to+\infty}
	\psi_B\int_0^1\left\{\tr_s\left[g\exp\left(-\left(	\widetilde{\mathbb{B}}_{u^2}'^H\right)^2 \right) \right]\right\}^{ds}ds.
	\end{multline}
Let $R^{\ker}=(\nabla^{\ker})^2$ be the curvature of the connection $\nabla^{\ker}$ 
on $\ker D_X^E$.
	As in (\ref{eq:1.16})-(\ref{eq:1.36}), modulo exact forms on $B$,
	\begin{multline}\label{eq:1.44}
	\tilde{\eta}_g(\pi',\mA)-\tilde{\eta}_g^{BC}(\pi',E\sqcup(\ker D_X^E)^{op})
	\\
	=\frac{1}{2\sqrt{\pi}\sqrt{-1}}\psi_B\int_0^{+\infty}e^{-t^2}dt\cdot \tr_s^V[g\mA
	\exp(-	R^{\ker}\oplus R^{\ker})].
	\end{multline}
	Note that 
	\begin{multline}\label{eq:1.45}
	\mA
	\exp(-	R^{\ker}\oplus R^{\ker})=
	\left(
	\begin{array}{cc}
	0&1\\
	1&0
	\end{array}
	\right)
	\left(
	\begin{array}{cc}
	\exp(-R^{\ker})&0\\
	0&\exp(-R^{\ker})
	\end{array}
	\right)
\\	=
	\left(
	\begin{array}{cc}
	0&\exp(-R^{\ker})\\
	\exp(-R^{\ker})&0
	\end{array}
	\right).
	\end{multline}
	Since the group action commutes with $\mA$ and $R^{\ker}$, $\tr_s^V[g\mA
	\exp(-	R^{\ker}\oplus R^{\ker})]=0$.
	So Proposition \ref{prop:1.05} follows from (\ref{eq:1.30}), 
	(\ref{eq:1.44}) and (\ref{eq:1.45}).

	The proof of Proposition \ref{prop:1.05} is completed.
\end{proof}  

\begin{rem}
As in (\ref{eq:1.35}), from (\ref{eq:0.08}) and (\ref{eq:1.24}), we have
	\begin{multline}
	d^B\tilde{\eta}_g^H(\pi',\mA)=d^B\tilde{\bar{\eta}}_g^{BC}(\pi',E)
	\\
	=\int_{X^g}\widehat{\mathrm{A}}_g(TX,\nabla^{TX})\ch_g(E,\nabla^E)
	-\ch_g(\ker D_X^E, \nabla^{\ker}),
	\end{multline}
	which is compatible with Proposition \ref{prop:1.05}.

\end{rem} 


\section{Anomaly formula}\label{s03} 

In this section, we will prove Theorem \ref{thm:0.04}. We use the notations and the assumptions in Section \ref{s0003}.

Let $s\in I$, $I=[0,1]$, parametrize 
a smooth path of equivariant horizontal subbundles
$\{T_{s}^HW\}_{s\in [0,1]}$.
The existence of such path follows from the affineness
of the space of equivariant splitting maps in the short
 exact sequence
$0\to TX\to TW\to \pi^*TB\to 0$.
Let $g_s^{TX}$ and $h_s^{E}$
be $G$-invariant metrics on $TX$ and $E$ connecting $(g_0^{TX}, h_0^E)$ and $(g_1^{TX}, h_1^E)$ smoothly. Let $\nabla_s^E$
be a $G$-invariant Hermitian connection 
on $(E,h_s^E)$ connecting $\nabla_0^E$ and $\nabla_1^E$ smoothly.
Now we get a smooth path connecting the geometric data
$(T_0^HW, g_0^{TX}, h_0^{E}, \nabla_0^{E})$ and 
$(T_1^HW,$ $ g_1^{TX}, h_1^{E}, \nabla_1^{E})$.
Let $D_{X,s}^E$ be the fiberwise Dirac operator associated 
with $(T_s^HW, g_s^{TX}, \nabla_s^{E})$. 

If $\ind(D_{X,0}^E)=0\in K_G^*(B)$, by the homotopy invariance 
of the equivariant family index, we have $\ind(D_{X,1}^E)=0\in K_G^*(B)$.
In this case, let $\mA_0$, $\mA_1$ be perturbation operators with respect to $D_{X,0}^E$, $D_{X,1}^E$ respectively.
Let $\tilde{\eta}_{g,0}(\pi,\mA_0)$ and $\tilde{\eta}_{g,1}(\pi,\mA_1)$ be corresponding 
equivariant eta forms with perturbations.
Let $P_0$, $P_1$ be orthogonal projections onto the eigenspaces of the positive spectrum of $D_{X,0}^E+\mA_0$, $D_{X,1}^E+\mA_1$ respectively.
Let $\mathrm{sf}_G\{(D_{X,0}^E+\mA_0, P_0), (D_{X,1}^E+\mA_1, P_1)\}\in K_G^*(B)$ be the equivariant Dai-Zhang higher spectral flow defined in \cite[Definitions 3.7, 3.8]{Liu21} (cf. also \cite[Definition 1.5]{DZ98}), which we 
denote by $\mathrm{sf}_G\{D_{X,0}^E+\mA_0, D_{X,1}^E+\mA_1\}$
for simplicity.
The following proposition was established in \cite{Liu21}, which is a generalization of \cite[Theorem 0.1]{DZ98}. 

	\begin{prop}\label{prop:2.01} \cite[Theorem 1.2]{Liu21}
	For any $g\in G$, modulo exact forms on $B$, we have
	\begin{multline}\label{D07}
	\tilde{\eta}_{g,1}(\pi, \mA_1)-\tilde{\eta}_{g,0}(\pi, \mA_0)=\int_{X^g}\widetilde{\widehat{\mathrm{A}}}_g(TX, \nabla_0^{TX}, \nabla_1^{TX})
	\, \ch_g(E, \nabla_0^{E})
	\\
	+\int_{X^g}\widehat{\mathrm{A}}_g(TX, \nabla_1^{TX} )\, \widetilde{\ch}_g(E, \nabla_0^{E},\nabla_1^{E})
	+
	\ch_g\left(\mathrm{sf}_G\{D_{X,0}^E+\mA_0, D_{X,1}^E+\mA_1\}\right).
	\end{multline} 
\end{prop}

\

\textbf{Case 1}: $\dim X$ is odd.

\

In this case, Theorem \ref{thm:0.04}
follows directly from Propositions \ref{prop:1.03} and 
\ref{prop:2.01}.

\begin{thm}\label{thm:2.02}
	If $\dim X$ is odd, for any $g\in G$, modulo exact forms on $B$, we have
	\begin{multline}
	\tilde{\bar{\eta}}_{g,1}^{BC}(\pi, E)-\tilde{\bar{\eta}}_{g,0}^{BC}(\pi, E)=\int_{X^g}\widetilde{\widehat{\mathrm{A}}}_g(TX, \nabla_0^{TX}, \nabla_1^{TX})
	\, \ch_g(E, \nabla_0^{E})
	\\
	+\int_{X^g}\widehat{\mathrm{A}}_g(TX, \nabla_1^{TX} )\, \widetilde{\ch}_g(E, \nabla_0^{E},\nabla_1^{E})
	+
	\ch_g\left(\mathrm{sf}_G\{D_{X,0}^E+P^{\ker}_0, D_{X,1}^E+P^{\ker}_1\}\right),
	\end{multline} 
	Here $P^{\ker}_0$, $P^{\ker}_1$ are the orthogonal 
	projections onto $\ker D_{X,0}^E\,$, $\ker D_{X,1}^E$
	respectively.
\end{thm}

Remark that if there is a smooth path connecting 
$(T_0^HW, g_0^{TX},h_0^E,\nabla_0^E)$ and 
$(T_1^HW, g_1^{TX},h_1^E,\nabla_1^E)$ such that
$\dim \ker D_{X,s}^E$ is locally constant, we have
$\mathrm{sf}_G\{D_{X,0}^E+P^{\ker}_0, D_{X,1}^E+P^{\ker}_1\}=0$ 
and $[\ch(\ker D_{X,0}^E)]=[\ch(\ker D_{X,1}^E)]$. 
This is the case in \cite[Theorem 1.2]{Liu17}.

\

\textbf{Case 2}: $\dim X$ is even.

\

In this case, recall that we assume  $\ker D_{X,i}^E=\ker D_{+,i}\oplus 
\ker D_{-,i}$, $i=0, 1$, are $\Z_2$-graded vector bundles over $B$.
Remark that 
if $s\in (0,1)$, $\dim \ker D_{X,s}^E$ is usually not locally constant
and $\ker D_{X,0}^E$ is not isomorphic to $\ker D_{X,1}^E$
as equivariant vector bundles.
We will use a perturbation trick in
\cite{Lo94,FL10} to overcome this difficulty.


Consider the equivariant fiber bundle $\pi\times \Id:W\times [0,1]\to B\times [0,1]$
and the equivariant $\Z_2$-graded Hilbert bundle
$\mathcal{H}_{\pm}$ over $B\times [0,1]$ with fiber 
$L^2(X_{(b,s)},\mathcal{S}_{\pm}(TX)\otimes E|_{X_{(b,s)}})$
for $(b,s)\in B\times [0,1]$.
Let $T^H\widetilde{W}|_{B\times\{s\}}:=T^H_sW\oplus \R$,
which is an equivariant horizontal subbundle with respect to $\pi\times \Id$.
Let $(g^{T\tilde{X}}, h^{\tilde{E}}, \nabla^{\tilde{E}})$
be the geometric data associated with $\pi\times\Id$ induced from $(g_s^{TX}, h_s^{E}, \nabla_s^E)$.
Let $\tilde{D}_X^E=\tilde{D}_+\oplus \tilde{D}_-$ 
be the corresponding $\Z_2$-graded fiberwise Dirac operator. 
Then $\tilde{D}_{\pm}|_{B\times \{0\}}=D_{\pm,0}$
and $\tilde{D}_{\pm}|_{B\times \{1\}}=D_{\pm,1}$.

By \cite[Lemma 7.13]{FL10} (see also \cite[p27]{Lo94}), there exist finite dimensional
 vector subbundles $L_{\pm}\subset \mathcal{H}_{\pm}$
 and complementary closed subbundles $K_{\pm}\subset
 \mathcal{H}_{\pm}$ over $B\times [0,1]$, i.e., $\mathcal{H}_{\pm}=L_{\pm}\oplus
 K_{\pm}$,
 such that $\tilde{D}_{\pm}\in \mathrm{Hom}(\mathcal{H}_{\pm},\mathcal{H}_{\mp})$
 is block diagonal as a map 
 $\tilde{D}_{\pm}:L_{\pm}\oplus K_{\pm}\to L_{\mp}\oplus K_{\mp}$
 and $\tilde{D}_{\pm}$ restricts to an isomorphism between
 $K_+$ and $K_-$. 
 Remark that the direct sums here may not be orthogonal.
 It is easy to see that we could choose $L_{\pm}$ and
 $K_{\pm}$ to be $G$-equivariant.

Note that
$\ker \tilde{D}_{\pm}\subset L_{\pm}$.
Add a new $\Z_2$-graded vector bundle 
$(L_+)_-\oplus (L_-)_+$ over $B\times [0,1]\times \{\mathrm{pt}\}$.
As in Section \ref{s02}, $(L_+)_-$ and $(L_-)_+$ are vector 
bundles $L_+$ and $L_-$ with the inverse grading.

As in Section \ref{s02},
we consider the equivariant fiber bundle $\pi'\times \Id:(W\sqcup B)
\times [0,1]\to B\times [0,1]$ with fiber $X\sqcup \{\mathrm{pt}\}$.
Let $\tilde{D}^{L}$  be the fiberwise Dirac operator on $\mathcal{H}^L=\mathcal{H}\oplus(L_+)_-\oplus(L_-)_+$.
Then $\tilde{D}^L=\tilde{D}_X^E\oplus 0\oplus 0$.
Let $P^L:\mathcal{H}^L\to L_+\oplus L_-\oplus (L_+)_-\oplus(L_-)_+$ 
be the projection.
Let 
\begin{align}
\mA^L=\left(
\begin{array}{cc}
0&1\\
1&0
\end{array}
\right)\circ P^L
\end{align}
on $(L_+\oplus(L_-)_+)\oplus((L_+)_-\oplus L_-)$.
From \cite[Lemma 7.20]{FL10}, $\tilde{D}^L+\mA^L$ is invertible.
Since $\mA^L$ anti-commutes with the $\Z_2$-grading of $\mH^L$, it
is a perturbation operator associated with $\tilde{D}^L$.

For $i=0,1$, $\ker D_{\pm, i}$ are equivariant subbundles of 
$L_{\pm,i}:=L_{\pm}|_{B\times \{i\}}$. Let $F_{\pm,i}$ be the orthogonal
complements of $\ker D_{\pm, i}$ in $L_{\pm,i}$.
Then $L_{\pm,i}=\ker D_{\pm,i}\oplus F_{\pm,i}$.
Since $D_{\pm,i}:F_{\pm,i}\to F_{\mp,i}$, we have $F_{+,i}\simeq F_{-,i}$ as equivariant vector bundles
over $B\times \{i\}$, $i=0,1$.
We identify $F_{+,i}$ and $F_{-,i}$ via the isomorphism $D_{+,i}$.
Let $P_i^F:\mathcal{H}^L|_{B\times\{i \}}\to \ker D_{X,i}^E \oplus (L_{+,i})_-\oplus(L_{-,i})_+$
be the orthogonal projection.
Let 
\begin{align}
\mA_{F,i}^{\ker}=\left(
\begin{array}{cccc}
0&1&0&0\\
1&0&0&0\\
0&0&0&1\\
0&0&1&0
\end{array}
\right)\circ P_i^F
\end{align}
on 
$$\big(\ker D_{+,i}\oplus (\ker D_{-,i})_+ \big)\oplus 
\big((\ker D_{+,i})_-\oplus \ker D_{-,i} \big)\oplus (F_{-,i})_+\oplus 
(F_{+,i})_- .$$
Let $D^L_i=\tilde{D}^L|_{B\times\{i\}}$, $i=0, 1$. 
Then $\mA_{F,i}^{\ker}$ are perturbation operators associated with $D^L_i $.

For $i=0,1$,
let $\nabla_i^{\mathcal{E}}$ and $\nabla_i^{\ker}$ be connections
in (\ref{eq:1.12}) and (\ref{eq:1.03}) associated with $(T_i^HW,$ $ g_{i}^{TX}, h_i^E, \nabla_i^E)$.
Let $\nabla_i^L$, $\nabla_i^F$ be the projected compatible connections
on $L_i$, $F_i$ respectively as in (\ref{eq:1.03}).
Let $\nabla_i^{\mathcal{E},F}=\nabla_i^{\mathcal{E}}\oplus 
\nabla_i^{\ker}\oplus\nabla_i^{F}$ on $\mathcal{H}^L|_{B\times\{i \}}$.
Set 
\begin{align}
\widetilde{\mathbb{B}}_{u^2,i}^F:=u(D_i^{L}+\chi(u)\mA_{F,i}^{\ker})+
\nabla_i^{\mathcal{E},F}
-\frac{c(T_i)}{4u},
\end{align}
where $c(T_i)$ is associated with $(T^H_iW, g_i^{TX})$.
Let $\tilde{\eta}_{g,i}^{H}(\pi', \mA_{F,i}^{\ker})$ be corresponding
eta forms for $g\in G$. 

For two equivariant geometric triples
$\underline{E}=(E,h^E,\nabla^E)$ and $\ul{F}=(F,h^F, \nabla^F)$,
if $E\simeq F$ as equivariant vector bundles over $B$, we denote 
by $\wi{\ch}_g(\ul{E},\ul{F})$ the equivariant Chern-Simons form
between $\ul{E}$ and $\ul{F}$ for $g\in G$. Moreover, we have
\begin{align}
d\, \wi{\ch}_g(\ul{E},\ul{F})=\ch_g(F,\nabla^F)-\ch_g(E,\nabla^E).
\end{align}

\begin{lemma}\label{lemma:2.03}
If $\dim X$ is even, for $g\in G$, $i=0,1$, we have
\begin{align}
\tilde{\eta}_{g,i}^{H}(\pi',\mA_{F,i}^{\ker})=\tilde{\bar{\eta}}_{g,i}^{BC}(\pi, E)-\widetilde{\ch}_g\left(\ul{F_{-,i}},\ul{F_{+,i}}\right) \mod d\Omega^*(B).
\end{align}
\end{lemma}
\begin{proof}
For $i=0,1$, let
\begin{align}
\mA_i=\left(
\begin{array}{cc}
0&1\\
1&0
\end{array}
\right)
\end{align}
on $\big(\ker D_{+,i}\oplus(\ker D_{-,i})_+\big)\oplus 
\big((\ker D_{+,i})_-\oplus \ker D_{-,i}\big)$.
Let 
\begin{align}
\mA_{F,i}=\left(
\begin{array}{cc}
0&1\\
1&0
\end{array}
\right)
\end{align}
on $(F_{-,i})_+\oplus 
(F_{+,i})_-$. Then $\mA_{F,i}^{\ker}=(\mA_i\oplus \mA_{F,i})\circ P_i^F$.
By Proposition \ref{prop:1.05}, we have
\begin{align}\label{eq:2.09}
\tilde{\eta}_{g}^{H}(\pi',\mA_{i})=\tilde{\bar{\eta}}_{g,i}^{BC}(\pi, E) \mod d\Omega^*(B).
\end{align}
From Definition \ref{defn:1.02}, (\ref{i18}) and (\ref{eq:1.25}), modulo exact forms on $B$,
\begin{align}\label{eq:2.10}
\tilde{\eta}_{g,i}^{H}(\pi',\mA_{F,i}^{\ker})=\tilde{\eta}_{g}^H(\pi',\mA_{i})-
\int_0^{\infty} \left.\frac{\sqrt{-1}}{2\pi}\tr_s\right.\left[g\left.\frac{\partial \nabla^F_{u^2,i}}{\partial u}\right.
\exp\left(-\frac{(\nabla^{F}_{u^2,i})^2}{2\pi \sqrt{-1}}\right)\right] du,
\end{align}
where $\nabla_{u,i}^F=\nabla_{+,i}^F\oplus \nabla_{-,i}^F+\sqrt{u}\chi(\sqrt{u})\mA_{F,i}$ on $F_{+,i}\oplus F_{-,i}$.
We denote by 
\begin{align}\label{eq:2.10a}
\hat{\eta}_g(\nabla^{F}_{u^2,i}):=\int_0^{\infty} \left.\frac{\sqrt{-1}}{2\pi}\tr_s\right.\left[g\left.\frac{\partial \nabla^F_{u^2,i}}{\partial u}\right.
\exp\left(-\frac{(\nabla^{F}_{u^2,i})^2}{2\pi \sqrt{-1}}\right)\right] du,
\end{align}
which is the equivariant eta form with perturbation when the fiber is a point and is the analogue of the analytic torsion for a finite dimensional complex \cite{BGSI}.
Let $\widetilde{\nabla}^F_{u,i}:=\nabla^{F}_{-,i}\oplus \nabla^F_{-,i}+\sqrt{u}\mA_{F,i}$
on $(F_{-,i})_+\oplus F_{-,i}$. We deform the metric and the connection from $\ul{F_{+,i}}\oplus\ul{F_{-,i}}$
	to $\ul{F_{-,i}}\oplus\ul{F_{-,i}}$. Since $\mA_{F,i}$ is invertible, by \cite[Theorem 2.10]{Bi98},
modulo exact forms, we have
\begin{align}\label{eq:2.10b}
\hat{\eta}_g(\nabla^{F}_{u^2,i})-\hat{\eta}_g(\wi{\nabla}^{F}_{u^2,i})
=\widetilde{\ch}_g\left(\ul{(F_{-,i})_+}\oplus\ul{F_{-,i}}, \ul{F_{+,i}}\oplus\ul{F_{-,i}} \right) 
		=\widetilde{\ch}_g\left(\ul{F_{-,i}}, \ul{F_{+,i}}\right).
\end{align}
From \cite[(2.24)]{Bi98}, we see that $\hat{\eta}_g(\wi{\nabla}^{F}_{u^2,i})=0$.
So
modulo exact forms, by (\ref{eq:2.10b}), we have
\begin{align}\label{eq:2.11}
\hat{\eta}_g(\nabla^{F}_{u^2,i})=\widetilde{\ch}_g\left(\ul{F_{-,i}}, \ul{F_{+,i}}\right).
\end{align}

From (\ref{eq:2.09}), (\ref{eq:2.10}), (\ref{eq:2.10a}) and (\ref{eq:2.11}), we obtain
Lemma \ref{lemma:2.03}.

\end{proof}

Since $L$ is an equivariant vector bundle over $B\times [0,1]$,
$L_{0}$ is isomorphic to $ L_{1}$ as equivariant $\Z_2$-graded
vector bundles over $B$.


%
\begin{thm}\label{thm:2.04}
		If $\dim X$ is even, for any $g\in G$, modulo exact forms on $B$, we have
	\begin{multline}
	\tilde{\bar{\eta}}_{g,1}^{BC}(\pi, E)-\tilde{\bar{\eta}}_{g,0}^{BC}(\pi, E)=\int_{X^g}\widetilde{\widehat{\mathrm{A}}}_g(TX, \nabla_0^{TX}, \nabla_1^{TX})
	\, \ch_g(E, \nabla_0^{E})
	\\
	+\int_{X^g}\widehat{\mathrm{A}}_g(TX, \nabla_1^{TX} )\, \widetilde{\ch}_g(E, \nabla_0^{E},\nabla_1^{E})
	+
	\ch_g\left(\mathrm{sf}_G\{D_0^L+\mA_{F,0}^{\ker}, D_1^L+\mA_{F,1}^{\ker}\}\right)
	\\
+\widetilde{\ch}_g(\ul{F_{-,1}},\ul{F_{+,1}})-\widetilde{\ch}_g(\ul{F_{-,0}},\ul{F_{+,0}})-\wi{\ch}_g(\ul{L_0}, \ul{L_1})
\\
+\wi{\ch}_g(L_1, \nabla_1^{\ker}\oplus\nabla_1^{F}, \nabla_1^L)-\wi{\ch}_g(L_0, \nabla_0^{\ker}\oplus\nabla_0^{F}, \nabla_0^L).
	\end{multline} 
%
\end{thm}
\begin{proof}
	For $i=0, 1$, 
	let $\nabla_i^{\mathcal{E},L}=\nabla_i^{\mathcal{E}}\oplus 
	\nabla_i^{L}$.
	Denote by $\mA^L_i=\mA^L|_{B\times \{i\}}$.
	Set 
	\begin{align}
	\widetilde{\mathbb{B}}_{u^2,i}^L:=u(D_i^{L}+\chi(u)\mA^L_i)+
	\nabla_i^{\mathcal{E},L}
	-\frac{c(T_i)}{4u}.
	\end{align}
Let $\tilde{\eta}_{g,i}^{H}(\pi',\mA^L_i)$ be corresponding equivariant eta forms.
Since $\mathrm{sf}_G\{D_0^L+\mA_{0}^{L}, D_1^L+\mA^L_1\}=0$,
from Proposition \ref{prop:2.01}, modulo exact forms, we have
	\begin{multline}\label{eq:2.14}
\tilde{\eta}_{g,1}^{H}(\pi',\mA^L_1)-\tilde{\eta}_{g,0}^{H}(\pi',\mA^L_0)=\int_{X^g}\widehat{\mathrm{A}}_g(TX, \nabla_0^{TX}, \nabla_1^{TX})
\, \ch_g(E, \nabla_0^{E})
\\
+\int_{X^g}\widetilde{\widehat{\mathrm{A}}}_g(TX, \nabla_1^{TX} )\, \widetilde{\ch}_g(E, \nabla_0^{E},\nabla_1^{E})
-\wi{\ch}_g(\ul{L_0}, \ul{L_1}),
\end{multline} 
and 
	\begin{multline}\label{eq:2.15}
\tilde{\eta}_{g,i}^{H}(\pi',\mA^L_i)-\tilde{\eta}_{g,i}^{H}(\pi',\mA_{F,i}^{\ker})=-\wi{\ch}_g(L_i, \nabla_i^{\ker}\oplus\nabla_i^{F}, \nabla_i^L)
\\
+\ch_g(\mathrm{sf}_G\{D_i^L+\mA_{F,i}^{\ker}, D_i^L+\mA^L_i\}).
\end{multline} 

Since 
\begin{multline}
\mathrm{sf}_G\{D^L+\mA_{F,0}^{\ker}, D^L+\mA_{F,1}^{\ker}\}=
\mathrm{sf}_G\{D^L+\mA_{F,0}^{\ker}, D^L+\mA^L_0\}
\\
-\mathrm{sf}_G\{D^L+\mA_{F,1}^{\ker}, D^L+\mA^L_1\},
\end{multline}
Theorem \ref{thm:2.04} follows from (\ref{eq:2.14}), (\ref{eq:2.15})
and Lemma \ref{lemma:2.03}.

The proof of Theorem \ref{thm:2.04} is completed.


\end{proof}

Remark that by \cite[Definition 2.14]{LiuMa20}, 
for any equivariant Chern-Simons form  $\wi{\ch}_g(\cdot)$,
$[0, \wi{\ch}_g(\cdot)]\in\widehat{K}_g^0(B) $.
So by \cite[Proposition 3.1]{LiuMa20},
$\wi{\ch}_g(\cdot)\in 
\ch_g(K_G^1(B))$\footnote{In \cite{LiuMa20}, it is easy to see that Proposition 3.1 holds for any compact Lie
group $G$ when $G$ acts on $B$ trivially.}.
Therefore we obtain Theorem \ref{thm:0.04}.

\

The following lemma will be used in Sections \ref{s04} and \ref{s05}.

\begin{lemma}\label{lemma:2.05}
Assume that $\dim X$ is even and
$\ind (D_X^E)=0\in K_G^0(B)$.
Let $\mA_X$ be a perturbation operator associated with $D_X^E$.
Then there exists $x\in K_G^1(B)$, which can be constructed as an equivariant higher spectral flow explicitly, such that modulo exact forms, for $g\in G$,
\begin{align}
\wi{\eta}_g(\pi,\mA_X)=\tilde{\bar{\eta}}_{g}^{BC}(\pi,E)
+\ch_g(x)	\mod d\Omega^*(B).
\end{align}
\end{lemma}\begin{proof}
Since $\ind (D_X^E)=0\in K_G^0(B)$,
by \cite[Proposition 2.4]{Se68},
there exists $n\in \N$, such that
$\ker D_{+}\oplus \varepsilon_+^n\simeq \ker D_{-}\oplus \varepsilon_-^n$,
where $\varepsilon_{\pm}^n$ are equivariant trivial $n$-dimensional
complex vector bundles over $B$ with $\Z_2$-grading.  
We identify $\ker D_{+}\oplus \varepsilon_+^n $ and $\ker D_{-}\oplus \varepsilon_-^n $ 
via this isomorphism.
We consider $\varepsilon_{\pm}^n$ as vector bundles over
$B\times \{\mathrm{pt} \}$ with trivial connections.

Let $P_{\varepsilon}: \mathcal{E}\oplus  \varepsilon_+^n\oplus \varepsilon_-^n 
\to \ker D_X^E\oplus  \varepsilon_+^n\oplus \varepsilon_-^n$ be the projection.
Let 
\begin{align}
\mA_{\varepsilon}^{\ker}=\left(
\begin{array}{cc}
0&1\\
1&0
\end{array}
\right)\circ P_{\varepsilon}
\end{align}
on $\left(\ker D_{+}\oplus \varepsilon_+^n \right)\oplus 
\left(\ker D_{-}\oplus \varepsilon_-^n \right) $.
Let 
\begin{align}
\mA_{\varepsilon}=\left(
\begin{array}{ccc}
\mA_{X}&0&0\\
0&0&1\\
0&1&0
\end{array}
\right)\circ P_{\varepsilon}
\end{align}
on $\mathcal{E}\oplus  \varepsilon_+^n\oplus \varepsilon_-^n$.
Then $\mA_{\varepsilon}^{\ker}$ and $\mA_{\varepsilon}$
are perturbation operators associated with the fiberwise Dirac operator
$D^{H}_X=D_{X}^E\oplus 0\oplus 0$ on $\mathcal{E}\oplus  \varepsilon_+^n\oplus \varepsilon_-^n$.
Let $\wi{\eta}_{g}^H(\pi',\mA_{\varepsilon}^{\ker})$ and $\wi{\eta}_{g}^H(\pi', \mA_{\varepsilon})$
be corresponding equivariant eta forms with perturbation.
By the proof of Proposition \ref{prop:1.05},
we have
\begin{align}\label{eq:2.20}
\wi{\eta}_{g}^H(\pi',\mA_{\varepsilon}^{\ker})=\tilde{\bar{\eta}}_{g}^{BC}(\pi,E)	\mod d\Omega^*(B).
\end{align}
From (\ref{eq:2.10}) and (\ref{eq:2.11}), since the connections
on $\varepsilon_{\pm}^n$ are trivial, modulo exact forms, we have
\begin{align}
\wi{\eta}_g(\pi,\mA_X)
=\wi{\eta}_{g}^H(\pi', \mA_{\varepsilon}).
\end{align}
From Proposition \ref{prop:2.01}, modulo exact forms,
we have
\begin{align}\label{eq:2.22}
\wi{\eta}_{g}^H(\pi', \mA_{\varepsilon})-\wi{\eta}_{g}^H(\pi',\mA_{\varepsilon}^{\ker})=\ch_g(\mathrm{sf}\{D^{H}_X+\mA_{\varepsilon}^{\ker}, D^{H}_X+\mA_{\varepsilon}\}).
\end{align}
So from (\ref{eq:2.20})-(\ref{eq:2.22}), modulo exact forms,
we have
\begin{align}
\wi{\eta}_g(\pi,\mA_X)-\ch_g(\mathrm{sf}\{D^{H}_X+\mA_{\varepsilon}^{\ker}, D^{H}_X+\mA_{\varepsilon}\})=\tilde{\bar{\eta}}_{g}^{BC}(\pi,E)	\mod d\Omega^*(B).
\end{align}

The proof of Lemma \ref{lemma:2.05} is completed.
\end{proof}


\section{Embedding formula}\label{s04}

In this section, we will prove Theorem \ref{thm:0.05}.
We use the notations and the assumptions in Section \ref{s0004}.

Let $\xi=\xi_+\oplus\xi_-$ be the $\Z_2$-graded vector bundle
over $W$.
From \cite[Proposition 3.6]{Embed}, if $\ind D_Y^{\mu}=0\in K_G^*(B)$,
then $\ind D_X^{\xi}=0\in K_G^*(B)$. Moreover, we have the following 
proposition, which is the equivariant family extension of the main result in \cite{BZ93}.

\begin{prop}\label{prop:3.01}\cite[Theorem 3.7]{Embed}
	Assume that $\ind D_Y^{\mu}=0\in K_G^*(B)$.
Let $\mA_Y$, $\mA_X$ be perturbation operators associated with $D_Y^{\mu}$, $D_X^{\xi}$ respectively. Then there exists a
perturbation operator $\mA_{\mA_Y,X}$ associated with $D_X^{\xi}$,
depending on $\mA_Y$,
such that modulo exact forms on $B$, for $g\in G$,
\begin{multline}
\tilde{\eta}_g(\pi_X, \mathcal{A}_X)=\tilde{\eta}_g(\pi_Y, 
\mathcal{A}_Y)+\int_{X^g}\widehat{\mathrm{A}}_g(TX, 
\nabla^{TX})\,\gamma_g^X(\mF_Y,\mF_X)
		\\
		+\int_{Y^g}\widetilde{\widehat{\mathrm{A}}}_g(
		\nabla^{TY,N},\nabla^{TX|_{W^g}})
		\,\widehat{\mathrm{A}}_g^{-1}(N,\nabla^{N})\,
		\ch_g(\mu,\nabla^{\mu})
\\
+\ch_g(\mathrm{sf}_G\{D_X^{\xi}+\mA_{\mA_Y,X}, D_X^{\xi} 
+\mA_{X}\})
.
\end{multline}
\end{prop}

Remark that in \cite{Embed}, the author assumed that the embedding
is totally geodesic. In fact, this condition is not necessary 
because we can use the variation formula Proposition \ref{prop:2.01}
to obtain the formula for the general embedding from the 
totally geodesic case. 

\

\textbf{Case 1}: $\dim X$ is odd. 

\

In this case, $\dim Y$ is also odd.
Let $P_Y^{\ker}$ and $P_X^{\ker}$ be orthogonal projections
onto $\ker D_Y^{\mu}$ and $\ker D_X^{\xi}$ respectively.
So from Proposition \ref{prop:1.03}, we have
\begin{align}
\begin{split}
&	\tilde{\eta}_g(\pi_Y,P_Y^{\ker})=\tilde{\bar{\eta}}^{BC}_g(\pi_Y, \mu) \mod d\Omega^*(B),
\\
&	\tilde{\eta}_g(\pi_X,P_X^{\ker})=\tilde{\bar{\eta}}^{BC}_g(\pi_X, \xi) \mod d\Omega^*(B).
	\end{split}
\end{align}

\begin{thm}
	If $\dim X$ is odd,
	modulo exact forms on $B$, for $g\in G$,
\begin{multline}
\tilde{\bar{\eta}}^{BC}_g(\pi_X, \xi)=\tilde{\bar{\eta}}^{BC}_g(\pi_Y, \mu)+\int_{X^g}\widehat{\mathrm{A}}_g(TX, 
\nabla^{TX})\,\gamma_g^X(\mF_Y,\mF_X)
\\
+\int_{Y^g}\widetilde{\widehat{\mathrm{A}}}_g(
\nabla^{TY,N},\nabla^{TX|_{W^g}})
\,\widehat{\mathrm{A}}_g^{-1}(N,\nabla^{N})\,
\ch_g(\mu,\nabla^{\mu})
\\
+\ch_g(\mathrm{sf}_G\{D_X^{\xi}+\mA_{P_Y^{\ker},X}, D_X^{\xi} 
+P_X^{\ker}\})
.
\end{multline}
\end{thm}

\

\textbf{Case 2}: $\dim X$ is even. 

\

In this case, $\dim Y$ is even.
We denote by $D_Y^{\mu}=D_{Y,+}\oplus D_{Y,-}$.
As in the setting of Proposition \ref{prop:1.05},
we add $B\times \{\mathrm{pt}\}$ and consider the embedding $W\sqcup B\to V\sqcup B$.
We denote by $(\ker D_{Y,-})_+$ and $ (\ker D_{Y,+})_-$
the vector bundles $\ker D_{Y,-}$ and $ \ker D_{Y,+}$
over $B\times \{\mathrm{pt}\}$ with the inverse grading
and assume that restricted on $B\times \{\mathrm{pt}\}$, the embedding is the identity map. 
This new embedding also forms an equivariant version of the Atiyah-Hirzebruch
direct image from $\mu\sqcup (\ker D_Y^{\mu})^{op}$ to
$\xi\sqcup (\ker D_Y^{\mu})^{op}$.

Let $\pi_W':W\sqcup B\to B$ be the fiber bundle induced from $\pi_W$.
Let $P^Y$ be the projection onto $\ker D_Y^{\mu}\oplus (\ker D_Y^{\mu})^{op}$.
Set 
\begin{align}
\mA_Y=\left(
\begin{array}{cc}
0&1\\
1&0
\end{array}
\right)\circ P^Y
\end{align}
on $\big(\ker D_{Y,+}\oplus(\ker D_{Y,-})_+\big)\oplus 
\big((\ker D_{Y,+})_-\oplus \ker D_{Y,-}\big)$.
Then
from Proposition \ref{prop:1.05}, modulo exact forms on $B$,
\begin{align}\label{eq:3.06}
\wi{\eta}_{g}^H(\pi_W',\mA_Y)=\tilde{\bar{\eta}}^{BC}_g(\pi_W, \mu).
\end{align}

Let $\pi_V':V\sqcup B\to B$ be the fiber bundle induced from $\pi_V$.
Let $\tilde{\eta}^{BC}_g(\pi_V', \xi\sqcup (\ker D_Y^{\mu})^{op})$ be the equivariant
Bismut-Cheeger eta form associated with this fiber bundle.
Since the equivariant Bismut-Cheeger eta form vanishes when the fiber is a point, as in (\ref{eq:1.30}),  we have
\begin{align}
\tilde{\eta}^{BC}_g(\pi_V, \xi)=
\tilde{\eta}^{BC}_g(\pi_V', \xi\sqcup (\ker D_Y^{\mu})^{op}).
\end{align}

Let $D^{H,\mu}$ and $D^{H,\xi}$ be fiberwise Dirac operators associated with $\pi_W'$ and $\pi_V'$. 
Since $\ind(D^{H,\mu})=0
\in K_G^0(B)$, $\ind(D^{H,\xi})=0\in K_G^0(B)$. 
Let $\mA_X$ be a perturbation operator associated with $D^{H,\xi}$.
From Lemma \ref{lemma:2.05}, there exists $x\in K_G^1(B)$, such that modulo exact forms,  
\begin{align}\label{eq:3.08}
\tilde{\eta}^{BC}_g(\pi_V', \xi\sqcup (\ker D_Y^{\mu})^{op})=\wi{\eta}_g(\pi_V',\mA_X)-\ch_g(x).
\end{align}
From (\ref{eq:0.10}), (\ref{eq:3.06})-(\ref{eq:3.08}) and Proposition \ref{prop:3.01},
we obtain Theorem \ref{thm:0.05} when $\dim X$ is even.

\begin{thm}
If $\dim X$ is even,
there exists $x\in K_G^1(B)$ such that modulo exact forms on $B$,
for $g\in G$,
\begin{multline}
\tilde{\bar{\eta}}^{BC}_g(\pi_V,\xi_+)-\tilde{\bar{\eta}}^{BC}_g(\pi_V, \xi_-)
=\tilde{\bar{\eta}}^{BC}_g(\pi_W,\mu)+\int_{X^g}\Ahat_g(TX, \nabla^{TX})\gamma_g(Y^g,X^g)
\\
+\int_{Y^g}\wi{\Ahat}_g(TY, \nabla^{TY,N},\nabla^{TX|_{W^g}})
\Ahat^{-1}_g(N,\nabla^N)\ch_g(E,\nabla^E)+\ch_g(x).
\end{multline}
\end{thm}


\section{Functoriality}\label{s05}

In this section, we will prove Theorem \ref{thm:0.06}. We use the notations and the assumptions
in Section \ref{s0005}.

From \cite[Lemma 3.6]{Liu21},
if $\Ind D_X^E=0\in K_G^*(V)$, then $\Ind D_Z^E=0\in K_G^*(B)$.
In this case, the functoriality of the equivariant eta forms with 
perturbation was obtained in \cite{Liu21}.
\begin{prop}\label{D08}\cite[Theorem 1.3]{Liu21}
	Let $\mA_Z$ and $\mA_X$ be perturbation operators associated with  $D_Z^E$ and $D_X^E$.
	Then there exists $x\in K_G^*(B)$, an equivariant higher spectral 
	flow, such that modulo exact forms on $B$, for $g\in G$,
	\begin{multline}\label{D09}
	\widetilde{\eta}_g(\pi_3, \mA_{Z})=\int_{Y^g}\widehat{\mathrm{A}}_g(TY, \nabla^{TY})\, \widetilde{\eta}_g(\pi_1,\mA_X)
	\\
	+\int_{Z^g}\wi{\widehat{\mathrm{A}}}_g(TZ, \nabla^{TY,TX},\nabla^{TZ})
	\, \ch_g(E, \nabla^{E})
	+ \ch_g(x).
	\end{multline}
	
\end{prop}

\

\textbf{Case} 1: $\dim X$ is odd and $\dim Y$ is even. 
Then $\dim Z$ is odd.

\

Since $\ker D_X^E$ and $\ker D_Z^E$ form equivariant vector bundles,  $\ind D_X^E=\ind D_Z^E=0\in K^1_G(B)$.
Let $P^{\ker D_X}$ and $P^{\ker D_Z}$ be orthogonal 
projections on $\ker D_X^E$ and $\ker D_Z^E$.
Then $P^{\ker D_X}$ and $P^{\ker D_Z}$ are perturbation
operators associated with $D_X^E$ and $D_Z^E$ respectively as in Section \ref{s02}.
By Proposition \ref{prop:1.03},
\begin{align}\label{eq:4.02}
\begin{split}
&\tilde{\eta}_g(\pi_1, P^{\ker D_X})=\tilde{\bar{\eta}}^{BC}_g(\pi_1,E)
\mod d\Omega^*(V),
\\
& \tilde{\eta}_g(\pi_3, P^{\ker D_Z})=\tilde{\bar{\eta}}^{BC}_g(\pi_3,E)
\mod d\Omega^*(B).
\end{split}
\end{align}

From Proposition \ref{D08} and (\ref{eq:4.02}), we obtain the functoriality
in this case.

\begin{thm}
	If $\dim X$ is odd and $\dim Y$ is even,
	there exists $x\in K_G^0(B)$, such that
modulo exact forms on $B$, for $g\in G$, 
\begin{multline}\label{eq:4.03}
\tilde{\bar{\eta}}^{BC}_g(\pi_3, E)=
\int_{Y^g}\widehat{\mathrm{A}}_g(TY, \nabla^{TY})\, \tilde{\bar{\eta}}^{BC}_g(\pi_1, E)
\\
+\int_{Z^g}\widetilde{\widehat{\mathrm{A}}}_g(TZ,\nabla^{TY,TX},\nabla^{TZ}  )
\, \ch_g(E, \nabla^{E})
+\ch_g(x).
\end{multline}
\end{thm}

\

\textbf{Case} 2: $\dim X$ is odd and $\dim Y$ is odd. Then $\dim Z$ is even.

\

In this case, $\ind D_X^E=0\in K_G^1(V)$ and $P^{\ker D_X}$ 
is a perturbation operator.
From \cite[Lemma 3.6]{Liu21}, $\ind D_Z^E=0\in K_G^0(B)$.
Let $\mA_Z$ be a perturbation operator with respect to $D_Z^E$.
Then by Lemma \ref{lemma:2.05}, there exists $x\in K_G^1(B)$,
such that modulo exact forms on $B$, for $g\in G$,
\begin{align}\label{eq:4.04}
\wi{\eta}_g(\pi_3,\mA_Z)=\tilde{\bar{\eta}}_{g}^{BC}(\pi_3,E)	+\ch_g(x) \mod d\Omega^*(B).
\end{align}
By Proposition \ref{prop:1.03},
\begin{align}\label{eq:4.05}
\tilde{\eta}_g(\pi_1, P^{\ker D_X})=\tilde{\bar{\eta}}^{BC}(\pi_1,E)
\mod d\Omega^*(V).
\end{align}


So from Proposition \ref{D08}, (\ref{eq:4.04}) and (\ref{eq:4.05}), we have
\begin{thm}
	If $\dim X$ and $\dim Y$ are both odd,
there exists $x\in K_G^1(B)$, such that
modulo exact forms on $B$, for $g\in G$, 
\begin{multline}\label{eq:4.06}
\tilde{\bar{\eta}}^{BC}_g(\pi_3, E)=
\int_{Y^g}\widehat{\mathrm{A}}_g(TY, \nabla^{TY})\, \tilde{\bar{\eta}}^{BC}_g(\pi_1, E)
\\
+\int_{Z^g}\widetilde{\widehat{\mathrm{A}}}_g(TZ,\nabla^{TY,TX},\nabla^{TZ}  )
\, \ch_g(E, \nabla^{E})+\ch_g(x).
\end{multline}
\end{thm}


\

\textbf{Case 3}: $\dim X$ is even and $\dim Y$ is odd. Then $\dim Z$ is odd.

\

Recall that $\ker D_X^E=\ker D_{X,+}\oplus \ker D_{X,-}$ is 
a $\Z_2$-graded equivariant vector bundle over $V$.
We consider a new equivariant fiber bundle $\pi_1':W\sqcup V\to V$
with fiber $X\sqcup \{\mathrm{pt}\}$ and with vector bundle 
$(\ker D_{X,-})_+\oplus (\ker D_{X,+})_-$ over $V\times \{\mathrm{pt}\}$.

Let $\wi{\eta}_{g}^H(\pi_1',\mA)$ be the equivariant eta form with perturbation $\mA$ in (\ref{eq:1.38}) on
$\big(\ker D_{X,+}\oplus(\ker D_{X,-})_+\big)\oplus 
\big((\ker D_{X,+})_-\oplus \ker D_{X,-}\big)$.
By Proposition \ref{prop:1.05}, we have
	\begin{align}\label{eq:4.07}
\tilde{\eta}_g^H(\pi_1',\mA)=\tilde{\bar{\eta}}_g^{BC}(\pi_1,E) 
\mod d\Omega^*(V).
\end{align}

We consider the equivariant fiber bundle $\pi_3\sqcup \pi_2^-:
W\sqcup V\to B$ such that the geometric triple over $W$ is $\ul{E}$
and that over $V$ is $\ul{(\ker D_X^E)^{op}}$. 
Since $\ker D_Y^{\ker D_X^E}$ and $\ker D_Z^E$ are equivariant vector bundles, from Definition \ref{defn:1.01},
\begin{align}\label{eq:4.08}
\wi{\eta}_{g}^{BC}(\pi_3\sqcup\pi_2^-,E\sqcup(\ker D_X^E)^{op} )=\wi{\eta}_{g}^{BC}(\pi_3, E)-\wi{\eta}_{g}^{BC}(\pi_2, \ker D_X^E).
\end{align}
Since $\dim Z$ and $\dim Y$ are both odd, by Proposition \ref{prop:1.03},
there exists a perturbation operator $\mA_P$ of the 
fiberwise Dirac operator associated with the fiber bundle
$\pi_3\sqcup \pi_2^-$, such that modulo exact forms, for $g\in G$,
\begin{multline}\label{eq:4.09}
\wi{\eta}_{g}(\pi_3\sqcup \pi_2^-, \mA_P)=\wi{\eta}_{g}^{BC}(\pi_3\sqcup\pi_2^-,E\sqcup(\ker D_X^E)^{op} )
\\
+\frac{1}{2}\ch_g(\ker D_Z^E)-\frac{1}{2}\ch_g(\ker D_Y^{\ker D_X^E}).
\end{multline}
So from Proposition \ref{D08} and (\ref{eq:4.07})-(\ref{eq:4.09}),
we have
\begin{thm}
	If $\dim X$ is even and $\dim Y$ is odd,
	there exists $x\in K_G^0(B)$, such that
modulo exact forms on $B$, for $g\in G$, 
\begin{multline}
\tilde{\bar{\eta}}^{BC}_g(\pi_3, E)=
\tilde{\bar{\eta}}^{BC}_g(\pi_2, \ker D_X^E)
+\int_{Y^g}\widehat{\mathrm{A}}_g(TY, \nabla^{TY})\, \tilde{\bar{\eta}}^{BC}_g(\pi_1, E)
\\
+\int_{Z^g}\widetilde{\widehat{\mathrm{A}}}_g(TZ,\nabla^{TY,TX},\nabla^{TZ}  )
\, \ch_g(E, \nabla^{E})+\ch_g(x).
\end{multline}

\end{thm}

\

\textbf{Case 4}: $\dim X$ is even and $\dim Y$ is even. Then $\dim Z$ is even.

\

We consider the fiber bundles in Case 3.
In this case, (\ref{eq:4.07}) and (\ref{eq:4.08}) still hold.
Let $\tilde{D}_X$ be the fibrewise Dirac operator with respect to $\pi_1'$. Then $\ind(\tilde{D}_{X})=0\in K_G^0(V)$.
By \cite[Lemma 3.6]{Liu21}, $\ind D_Z^E-\ind D_Y^{\ker D_X^E}
=0\in K_G^0(B)$.
By Lemma \ref{lemma:2.05}, in this case, (\ref{eq:4.09}) is replaced 
by
\begin{align}
\wi{\eta}_{g}(\pi_3\sqcup \pi_2^-, \mA_P)=\wi{\eta}_{g}^{BC}(\pi_3\sqcup\pi_2^-,E\sqcup(\ker D_X^E)^{op} )
+\ch_g(x).
\end{align}
Here $x$ is an element of $K_G^1(B)$. 
\begin{thm}
	If $\dim X$ and $\dim Y$ are both even,
	there exists $x\in K_G^1(B)$, such that
modulo exact forms on $B$, for $g\in G$, 
\begin{multline}
\tilde{\bar{\eta}}^{BC}_g(\pi_3, E)=
\tilde{\bar{\eta}}^{BC}_g(\pi_2, \ker D_X^E)
+\int_{Y^g}\widehat{\mathrm{A}}_g(TY, \nabla^{TY})\, \tilde{\bar{\eta}}^{BC}_g(\pi_1, E)
\\
+\int_{Z^g}\widetilde{\widehat{\mathrm{A}}}_g(TZ,\nabla^{TY,TX},\nabla^{TZ}  )
\, \ch_g(E, \nabla^{E})+\ch_g(x).
\end{multline}
\end{thm}

	\vspace{3mm}\textbf{Acknowledgements}\ \
The author would like to thank Professors Xiaonan Ma and Weiping Zhang
for helpful discussions.
He would also like to thank the anonymous referee for various 
useful suggestions that improved the clarity and quality of the paper.
This research is partly supported by  Science and Technology Commission 
of Shanghai Municipality (STCSM), grant No.18dz2271000, Natural Science Foundation of Shanghai, grant No.20ZR1416700 and NSFC No.11931007.

\def\cprime{$'$}
\providecommand{\bysame}{\leavevmode\hbox to3em{\hrulefill}\thinspace}
\providecommand{\MR}{\relax\ifhmode\unskip\space\fi MR }
\providecommand{\MRhref}[2]{%
  \href{http://www.ams.org/mathscinet-getitem?mr=#1}{#2}
}
\providecommand{\href}[2]{#2}

\end{document}